\documentclass[11pt]{article} 

\usepackage{geometry} 
\usepackage{amsmath,amssymb,amsthm,calrsfs}
\usepackage{amsrefs}
\usepackage{enumerate}
\usepackage{aliascnt}
\usepackage{extarrows}
\usepackage[symbol]{footmisc}
\usepackage{subfig}
\usepackage{multirow}
\usepackage{tikz}
\usetikzlibrary{calc}
\usepackage{graphicx}
\usepackage{xcolor}
\usepackage[affil-it]{authblk}
\usepackage[mathscr]{eucal}
\usepackage{fullpage}
\usepackage{bbm}
\usepackage[normalem]{ulem}

\usepackage[colorlinks,linkcolor=blue,citecolor=red, bookmarksopen=false]{hyperref}

\def\NewTheorem#1#2{%
	\newaliascnt{#1}{thmm}
	\newtheorem{#1}[#1]{#2}
	\aliascntresetthe{#1}
	\expandafter\def\csname #1autorefname\endcsname{#2}
}

\newcommand{\ubar}[1]{\mkern3mu\underline{\mkern-3mu #1\mkern-3mu}\mkern3mu}

\newcommand{\spr}[1] {{\odot_{ #1 }} }
\newcommand{\pr}[1] {{\otimes_{ #1 }} }

\newcommand{\lwrX}{\ubar{\mathbf{X}}}
\newcommand{\lwrY}{\ubar{\mathbf{Y}}}

\newcommand{\uprX}{\bar{\mathbf{X}}}

\numberwithin{equation}{section}

\NewTheorem{lem}{Lemma}
\NewTheorem{thm}{Theorem}
\NewTheorem{cor}{Corollary}
\NewTheorem{prop}{Proposition}
\theoremstyle{definition}
\NewTheorem{defi}{Definition}
\theoremstyle{remark}
\NewTheorem{rem}{Remark}
\theoremstyle{definition}
\NewTheorem{ex}{Example}

\makeatletter
\def\namedlabel#1#2{\begingroup
   \def\@currentlabel{#2}%
   \label{#1}\endgroup
}
\makeatother
\newcommand{\tree}{{\mathbb{T}}}

\providecommand{\keywords}[1]{\textbf{Keywords.} #1}
\providecommand{\amssc}[1]{\textbf{AMS subject classification.} #1}

\title{On Le Jan-Sznitman's stochastic approach to the Navier-Stokes equations}
\author{Radu Dascaliuc\footnote{Department of Mathematics,  Oregon State University, Corvallis, OR, 97331. radu.dascaliuc@oregonstate.edu} \hspace{0.5in} Tuan N.\ Pham\footnote{Department of Mathematics,  Brigham Young University, Provo, UT, 84602. tuan.pham@mathematics.byu.edu} \hspace{0.5in} Enrique Thomann\footnote{Department of Mathematics,  Oregon State University, Corvallis, OR, 97331. enrique.thomann@oregonstate.edu}}
\date{\today}
\begin{document}
\maketitle
\begin{abstract} 
The paper explores the symbiotic relation between the Navier-Stokes equations and the associated stochastic cascades. Specifically, we examine how some well-known existence and uniqueness results for the Navier-Stokes equations can inform about the probabilistic features of the associated stochastic cascades, and how some probabilistic features of the stochastic cascades can, in turn, inform about the existence and uniqueness (or the lack thereof) of solutions. Our method of incorporating the stochastic explosion gives a simpler and more natural method to construct the solution compared to the original construction by Le Jan and Sznitman. This new stochastic construction is then used to show the finite-time blowup and non-uniqueness of the initial value problem for the Montgomery-Smith equation. 
We exploit symmetry properties inherent in our construction to give a simple proof of the global well-posedness results for small initial data in scale-critical Fourier-Besov spaces. We also obtain the pointwise convergence of the Picard's iteration associated with the Fourier-transformed Navier-Stokes equations. 
\end{abstract}
.\\\\
\amssc{35C15, 35Q30, 60J80, 76B03}.

\section{Introduction}\label{introduction}

The paper focuses on a natural connection, first noticed by Le Jan and Sznitmann \cite{lejan}, between the deterministic 3-dimensional incompressible Navier-Stokes equations and the branching stochastic cascades. Thanks to this connection, the problems of well-posedness can be recast in terms of probabilistic properties of a random functional defined on the stochastic cascades. These cascades provide interesting insights on the way information propagates from the initial data to the solution at time $t$, and shed some light on the standing problems in the regularity theory. 

Our main purpose is to explore the symbiotic relation between the Navier-Stokes equations and the associated stochastic cascades. On one hand, we examine how some well-known existence and uniqueness results for the Navier-Stokes equations can inform about the probabilistic features such as the integrability and explosion of the associated stochastic cascades. On the other hand, we investigate how some probabilistic features of these cascades such as the stochastic explosion and the distribution of number of crossing branches, in turn, can inform about the existence and uniqueness (or the lack thereof) of solutions.

One challenge in this direction is to interpret well-posedness results in typical functional settings used in the literature through the lens of this probabilistic structure, which naturally favors weighted $L^\infty$ spaces with very specific weights \cite {rabi}. Another challenge is to incorporate stochastic explosion in Le Jan-Sznitman's approach. We resolve these issues by using a ``minimal'' solution constructed from the stochastic cascade\textemdash an idea adapted from \cite{alphariccati}\textemdash and establishing the connection between this solution and the Picard's iterations. We now proceed to lay out the framework.

Consider the Cauchy problem for the incompressible Navier-Stokes equations in $d$-dimensions:
\begin{equation}\label{NS}\tag{NS}\left\{ {\begin{array}{*{20}{rcl}}
   {{\partial _t}u - \Delta u + u\cdot \nabla u + \nabla p = 0}&~~{\rm in}&\mathbb{R}^d\times(0,\infty),  \\
   {{\rm div}\, u = 0}&~~{\rm in}&\mathbb{R}^d\times(0,\infty),  \\
   {u(\cdot,0) = {u_0}}&~~{\rm in}&\mathbb{R}^d.  \\
\end{array}} \right.\end{equation}
The system has a well-known scaling property 

\begin{equation}\label{scaling}
u(x,t)\to\lambda u(\lambda x,\lambda^2 t),~~p(x,t)\to\lambda^2 p(\lambda x,\lambda^2 t),~~u_{0}(x)\to\lambda u_0(\lambda x),\ \ \lambda\in\mathbb{R}.
\end{equation}
The classic Kato's mild solutions are the solutions to the integral equation
\begin{equation}\label{eq:828191}u(t)={{e}^{t\Delta }}{{u}_{0}}-\int_{0}^{t}{{{e}^{(t-s)\Delta }}\mathbf{P}\operatorname{div}(u\otimes u)ds}\end{equation}
obtained by Banach fixed-point method.
Here $\mathbf{P}$ denotes the Leray projection onto the divergence-free vector fields. 
Historically, the regularity theory of mild solutions traces back to the pioneering work of Leray \cite{leray}. 
Local well-posedness is known in various scale-subcritical and scale-critical spaces of the initial data.
 Global well-posedness is also known in various critical spaces provided that the initial data is sufficiently small.
Readers can refer to \cite{lemarie2002, lemarie2016, bahouri} for surveys of such results. Key to obtaining mild solutions is to select suitable function spaces for the initial data and the solutions (called the adapted space and path space 
\cite[p.\ 146]{lemarie2002}) so that the fixed point method works. 
This paper will focus on the Fourier formulation of the Navier-Stokes equations, i.e.\ the 
Fourier transform of \eqref{eq:828191}:
\begin{equation*}\label{FNS}
\tag{FNS}
{v}=U({v}_0)+{B}({v},{v})
\end{equation*}
where $v=\mathscr{F}\{{u}\}$, $v_0=\mathscr{F}\{{u}_0\}$, $U(v_0)=e^{-t|\xi|^2}v_0$ and
\begin{equation}\label{eq:71214}
{B}(v,v)=c_0\int_{0}^{t}{{{e}^{-s|\xi {{|}^{2}}}}|\xi |\int_{{{\mathbb{R}}^{d}}}{v(\eta ,t-s)\pr{\xi}v(\xi -\eta ,t-s)d\eta ds}}.
\end{equation}

Here $c_0=(2\pi)^{-d/2}$ and $\mathscr{F}\{u\}=c_0\int_{\mathbb{R}^d}u(x,t)e^{-ix\cdot\xi}dx$. The Leray projection is encoded in the $\pr{}$-product, which is a non-commutative, non-associative vector operation satisfying $|a\pr{\xi} b|\le|a||b|$. Specifically, $a\pr{\xi} b=-i(e_\xi\cdot b)(\pi_{\xi^\perp}a)$ where $e_\xi=\xi/|\xi|$ and $\pi_{\xi^\perp}a=a-(e_\xi\cdot a)e_\xi$. 

The existence and uniqueness results for \eqref{FNS} can be obtained by a fixed-point-type argument in suitably defined functional settings (adapted spaces) including the weighted $L^2$ spaces and Herz spaces $\dot{K}^\alpha_{p,q}$, which are the Fourier transform of the Fourier-Besov spaces $F\dot{B}^{\alpha}_{p,q}$; see e.g.\ 
\cite{cannonekarch, linlei2011, cannonewu, konieczny, xiao2014}, \cite[Sec.\ 16.3]{lemarie2002}, \cite[Sec.\ 8.7]{lemarie2016}. The formulation \eqref{FNS} is naturally associated with the Picard's iteration  $v^{(1)} = U(v_0)$ and $v^{(n+1)}= U(v_0) + B(v^{(n)},v^{(n)})$. Typically, the convergence of the iteration in a chosen path space is obtained by showing the boundedness of the bilinear operator $B$ in that space.  
From the convergence in norm, together with standard functional analysis tools, one might infer that there exists a \emph{subsequence} of $v^{(n)}$ that converges pointwise. The pointwise convergence of the \emph{entire} Picard's iteration, although often not required for regularity theory purposes, is desirable in applications (for example, in validating a numerical simulation).

The first goal of the paper is to establish conditions for the pointwise convergence of the sequence $v^{(n)}$ for initial data $v_0$ in suitable adapted spaces including the scale-critical Herz spaces (\autoref{75211}) and explore the connection between convergence of Picard iterations and the probabilistic iterative processes naturally associated with \eqref{FNS}. In fact, we will use this probabilistic structure in a crucial way to obtain our convergence results.

The idea of interpreting the solutions to a deterministic evolutionary PDE as the expectation of an associated stochastic process goes back to the classical relation between the heat equation and Brownian motion, which was extended to the Kolmogorov-Petrovskii-Puskinov (KPP) equation in the seminal work of McKean in 1975 \cite{mckean}. In particular, the presence of nonlinearity in the KPP equation leads to a branching stochastic structure (branching Brownian motion), which serves as a blueprint for applications of branching processes in the analysis of semilinear parabolic equations. 

For the 3-dimensional Navier-Stokes equations, Le Jan and Sznitmann \cite{lejan} noticed that after rescaling \eqref{FNS} by a suitable kernel $h(\xi)$, the solution can be interpreted as an expectation of a random variable\textemdash a \emph{solution process}\textemdash built on a branching stochastic structure\textemdash a {\em stochastic cascade}. 
This solution process is well-defined provided that the stochastic cascade does not generate infinitely many branches in finite time\textemdash a phenomenon called {\em stochastic non-explosion}. In Le Jan-Sznitmann's work, the non-explosion was obtained by a somewhat artificial thinning procedure. Their approach was later generalized to any spacial dimension $d\geq 3$ by Bhattacharya et al.\ \cite{rabi}. 
In contrast to various existing shell models of turbulence which are based on  statistical assumptions on
the dynamics of fluid flows, Le Jan-Sznitman's
stochastic cascade provides a precise notion of averaging in the frequency domain directly from the Navier-Stokes equations, making it an appealing tool to investigate the energy spectrum of turbulent flows. We do not pursue this direction in the present paper.

%
Dascaliuc et al.\ \cite{chaos} simplified the construction of the stochastic process to address {\em the stochastic explosion}\textemdash the phenomenon that the stochastic cascade produces infinitely many branches in finite time. Recently, it was shown in \cite{part2} that the stochastic cascade compatible with the natural scaling \eqref{scaling} is explosive in dimension $d=3$ but not in dimensions $d\ge 12$. In the literature of Markov processes, stochastic explosion has been  exploited to produce blowup or nonuniqueness results of the associated  Kolmogorov backward equations (see e.g.\ \cite[Ch.\ 4, Sec.\ 6]{edbook}). Incorporating the stochastic explosion of a \emph{branching} Markov structure into the construction of solutions to the associated differential equation has been done in the case study of the $\alpha$-Riccati equation. In this model, the stochastic explosion was used to as a mechanism for finite-time blowup and nonuniqueness of solutions \cite{alphariccati}. 


The second goal of the paper is to employ ideas from \cite{alphariccati} to introduce a simpler and more natural construction of solutions to \eqref{FNS} from the associated stochastic cascades that eliminates the necessity of thinning and allows for the stochastic explosion (\autoref{819191}). Our method results in the notion of \emph{minimal cascade solutions}, which is the counterpart of the minimal solutions of $\alpha$-Riccati equation developed in \cite{alphariccati} and is reminiscent of the notion of minimal solutions in the context of Markov processes.  Interestingly, incorporating explosion in the stochastic approach naturally connects with the aforementioned Picard's iterations, providing a natural venue to study their convergence, which was essential in establishing \autoref{75211}. We also note that the elimination of thinning improves the estimates on the size of initial data that ensures a global-in-time solution to \eqref{FNS}.

Because the solution to \eqref{FNS} is given by the expectation of a solution process, another obstacle to the stochastic construction besides the stochastic explosion is the possible lack of integrability of the solution process. 
To secure the integrability, Le Jan-Sznitman and other authors, e.g.\ \cite{rabi, orum}, imposed a {\em pointwise} smallness condition on the normalized initial data. Although such a condition was enough to guarantee the existence of a global solution, the  local existence of solutions for large initial data was left unresolved by the stochastic approach (except by modifying the standard majorizing kernel, which complicates the stochastic process \cite[Ch.\ 6]{orum}, \cite{chris2d, rabi}).

The third goal of the paper is to establish the integrability of the solution process for initial data in functional spaces commonly used in the analysis of the \eqref{FNS}, thereby relaxing the pointwise smallness condition and obtaining local solutions for large initial data (\autoref{75212}). In order to do so, we establish a majorization principle 
that allows us to compare the size of solution process built for \eqref{FNS} to that built for a simplified equation \eqref{MS} introduced by Montgomery-Smith \cite{smith}. The simplified product structure of the Montgomery-Smith equation yields simple symmetry properties for the corresponding solution process. Interestingly, these simple symmetry properties lead to an alternative proof for the global well-posedness of \eqref{NS} in the scale-critical Fourier-Besov spaces (\autoref{simpleproof})\textemdash a result shown in \cite{konieczny, xiao2014, lizheng, lemarie2016}.

We further illustrate the versatility our stochastic approach that incorporates explosion by giving simple proofs for both the finite-time blowup (recovering the result in \cite{smith}) and the nonuniqueness of initial value problem for the Montgomery-Smith equation (\autoref{nonuniqueness}). Although similar conclusions for Navier-Stokes equation remain elusive due to possible depletions of nonlinearity. Such cancellations come from a vectorial product encoding both the nonlinearity and the divergence-free constraint of \eqref{NS}. By examining the geometric structure of the product, we obtain a cancellation-type property (\autoref{sprod_lem}) that improves the size the initial data leading to global-in-time solutions.

The organization of the paper is as follows. 
In \autoref{setup}, we give preliminaries regarding the stochastic cascade setup for \eqref{FNS}.
We define the minimal cascade solution to \eqref{FNS} that eliminates the thinning procedure and allows for stochastic explosion. We   
show, under rather mild assumptions, the equivalence between the minimal cascade solution and the \emph{thinned cascade solution} defined by \cite{lejan, rabi}. These two solutions will be referred to as \emph{cascade solutions} due to the nature of their construction. We also show that the cascade solutions agree with the Picard's iteration scheme of \eqref{FNS}. 
In \autoref{reduction}, we analyze the relation between the Navier-Stokes equations and the Montgomery-Smith equation which have exactly the same underlying stochastic cascade structure. 
We strengthen the majorizing principle used in \cite{lejan, rabi} by exploiting the geometry of the nonlinear term and some natural symmetry properties of the solution process. 
In \autoref{applications}, we give some applications of the generalized majorizing principle.
An analytic proof of the generalized majorizing principle is given in the appendix.
\section{Stochastic cascade setup and minimal cascade solution}\label{setup} 
In this section, we describe the stochastic cascade structure associated with \eqref{FNS}. In particular, we modify the construction from \cite{lejan} using ideas from \cite{chaos, alphariccati} with the aim to incorporate explosive cascades. 

First, in order to optimize our estimates in subsequent sections,  we symmetrize our bilinear term. Note that $B(u,u)$ in \eqref{eq:71214} is invariant under the change of variable $\tilde{\eta}=\xi-\eta$.  Thus, we can view $B(u,u)$ as the diagonal part of the bilinear operator

\begin{equation}\label{eq:71214-2}
{B}(f,g)=\frac{1}{2}(B(f,g)+B(g,f))=c_0\int_{0}^{t}{{{e}^{-s|\xi {{|}^{2}}}}|\xi |\int_{{{\mathbb{R}}^{d}}}{f(\eta ,t-s)\spr{\xi}g(\xi -\eta ,t-s)d\eta ds}},
\end{equation}
where $\spr{}$ is the symmetrized version of $\pr{}$:
\begin{equation}\label{symm-odot}
a\spr{\xi}b= \frac{1}{2}(a\pr{\xi} b+b\pr{\xi} a)=\frac{-i}{2}\left( (e_\xi\cdot b)(\pi_{\xi^\perp}a)+ (e_\xi\cdot a)(\pi_{\xi^\perp}b)\right).
\end{equation}

From now on, all our discussion will use this symmetrized form of the nonlinearity. The symmetrized product $\spr{}_\xi$ satisfies the following estimate.

\begin{lem}\label{sprod_lem}
Let $\xi\in\mathbb{R}^d\backslash\{0\}$. Then
\begin{equation}\label{sprod-bd}
|v\spr{\xi} w|\le\frac{1}{2}|v|\,|w|\ \ \ \forall v,w\in \mathbb{C}^d.
\end{equation}
Moreover, the constant factor 1/2 on the right hand side is optimal.
\end{lem}

\begin{proof}
Using the orthogonal decomposition $\mathbb{C}^d=\mathbb{C}\xi\oplus\xi^\perp$, we can write
\[v=\alpha \xi+ a, \qquad w=\beta \xi + b,\]
where $\alpha,\beta\in\mathbb{C}$ and $a=\pi_{\xi^{\perp}}v$, $b=\pi_{\xi^{\perp}}w$. Then,
\[|2 v\spr{\xi}w|^2=\left| \alpha b +\beta a\right|^2\le (|\alpha||b|+|\beta||a|)^2\le\left[\left(|\alpha|^2+|a|^2\right)^{1/2}\left(|b|^2+|\beta|^2\right)^{1/2}\right]^2=|v|^2|w|^2.\]
%
To get the equality in \eqref{sprod-bd}, one can choose, for example, unit vectors $v,w\in\mathbb{C}^d$ satisfying
$b=\gamma a$, $\alpha=\gamma/(1+\gamma^2)^{1/2}$, and $\beta=1/(1+\gamma^2)^{1/2}$
for some $\gamma\ge0$.
\end{proof}

We now turn to the construction of stochastic cascades associated with \eqref{FNS}. The key idea is to normalize $v$ by $\chi=c_0v/h$, where $h:\mathbb{R}^d\backslash\{0\}\to(0,\infty)$ is chosen such that, for any $\xi\in \mathbb{R}^d\backslash\{0\}$,
$H(\eta |\xi )=\frac{h(\eta )h(\xi -\eta )}{|\xi |h(\xi )}$
defines a probability density function in $\eta\in\mathbb{R}^d$ \cite{lejan, rabi, chaos}. For this, $h$ must satisfy $h*h=|\xi|h$ and is called a \emph{standard majorizing kernel} \cite{rabi}. Such functions include the Bessel kernel $h_\text{b}(\xi)=c{{{e}^{-|\xi |}}}{|\xi |}^{-1}$ for $d=3$, and the scale-invariant kernel $h_{\text{in}}(\xi)=c_d|\xi|^{1-d}$ for $d\ge 3$ (\cite[p.\ 18-19]{orum}).
Standard majorizing kernels do not exist for dimensions $d<3$ \cite{chris-mina}. 
With the normalization above, \eqref{FNS} can be written in terms of $\chi$ as
\begin{equation}\label{nFNS}\tag{nFNS} 
\chi (\xi ,t)={{e}^{-t|\xi {{|}^{2}}}}{{\chi }_{0}}(\xi)+\int_{0}^{t}{{{e}^{-s|\xi {{|}^{2}}}}|\xi {{|}^{2}}\int_{{{\mathbb{R}}^{d}}}{\chi (\eta ,t-s)\odot_{\xi}\chi (\xi -\eta ,t-s)H(\eta |\xi )d\eta ds }}.\end{equation}
At this point, the solution $\chi$ can be interpreted probabilistically as the expected value of a ``solution process'' $\mathbf{X}$ defined implicitly by
\begin{equation}\label{eq:822191}\mathbf{X}(\xi,t)=\left\{ \begin{array}{*{35}{l}}
   \chi_0(\xi) & \text{if} & {{T}_{0}}\ge t,  \\
   {\mathbf{X}}^{(1)}(W_1, t-{{T}_{0}}){{\odot }_{\xi }}{\mathbf{X}}^{(2)}(W_2, t-{{T}_{0}}) & \text{if} & {{T}_{0}}< t.  \\
\end{array} \right.\end{equation}
Here $T_0$ is an exponentially distributed random variable with mean $|\xi|^{-2}$, $W_1\in\mathbb{R}^d$ is a random variable independent of $T_0$ with probability density $H(\cdot|\xi)$, $W_2=\xi-W_1$, and $\mathbf{X}^{(1)}$ and $\mathbf{X}^{(2)}$ are two independent copies of $\mathbf{X}$. The definition \eqref{eq:822191} when applied recursively induces two families of random variables $\{{T}_{v}\}_{v\in\tree}$ and $\{W_{v}\}_{v\in\tree}$, indexed by a binary tree 
$\tree=\{0\}\cup(\cup_{n=1}^\infty \{1,2\}^n)$, described by
\begin{enumerate}[(i)]
\item $W_0=\xi,$\vspace{-4pt}
\item Given $W_v$, $T_v$ is distributed as Exp($|W_v|^{2}$), $W_{v1}$ is distributed as $H(\cdot|W_v)$ and $W_{v2}=W_v-W_{v1}$,\vspace{-4pt}
\item Given $W_{v1}$ and $W_{v2}$, the sub-families $\{W_{v1\sigma}\}_{\sigma\in\mathbb{T}}$ and $\{W_{v2\sigma}\}_{\sigma\in\mathbb{T}}$ are independent of each other.
\end{enumerate}
\begin{figure}[h]
\centering
\includegraphics[scale=.8]{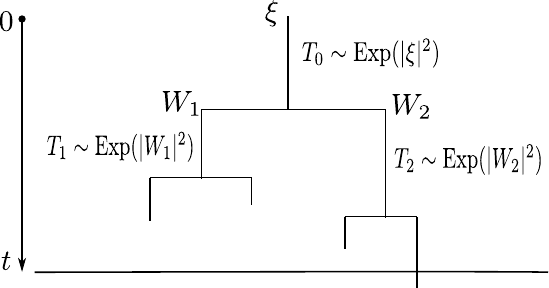}
\caption{Cascade figure that illustrates $\{{T}_{v}\}_{v\in\tree}$ and $\{W_{v}\}_{v\in\tree}$.}
\label{cascade}
\end{figure}
\autoref{cascade} is a cascade figure that illustrates the families $\{{T}_{v}\}_{v\in\tree}$ and $\{W_{v}\}_{v\in\tree}$ induced from the recursion \eqref{eq:822191}. For each $\xi\in\mathbb{R}^d\backslash\{0\}$, these random variables are defined on a probability space $(\Omega, \mathcal{F}, \mathbb{P}_\xi)$. We  denote by $\mathbb{E}_\xi$ the expectation of random variables with respect to the probability measure $\mathbb{P}_\xi$.  Note that the recursion \eqref{eq:822191} might keep on going indefinitely, making $\mathbf{X}$ not well-defined, a phenomenon referred as {\it{stochastic explosion}} or simply explosion. In fact, it has been shown in \cite{part2} that the cascade corresponding to the scale-invariant kernel in $\mathbb{R}^3$ is almost surely explosive for every $\xi\in\mathbb{R}^3\backslash\{0\}$. 

The \emph{explosion time} of the cascade starting at $\xi$ is a random variable $\zeta$  defined as (see also \cite{alphariccati, chaos, part1})
\begin{equation}
\label{explosiontime}
\zeta=\zeta_{\xi} 
=\underset{n\ge 0 }{\mathop{\sup }}\,\underset{|v|=n}{\mathop{\min }}\,\sum\limits_{j=0}^{n} T_{v|j}.\end{equation}
Note that in the {\em non-explosion} event $[\zeta> t]$, the recursion \eqref{eq:822191} terminates in finitely many steps. In the {\em explosion} event $[\zeta<t]$, the solution process $\mathbf{X}$ is not well-defined for $t>\zeta$. 

A stochastic cascade where $\zeta=\infty$ a.s. is called {\em non-explosive}. For non-explosive cascades, $\mathbf{X}$ is well defined a.s. for all $t>0$. 
In fact, in this case, $\mathbf{X}$ is a $\spr{}$-product, in a suitable order, of the initial data $\chi_0$ evaluated on the leaves of the cascade tree that ``reached" time $t$:
\begin{equation}\label{Vxit}
V(\xi,t)=\left\{v\in\mathbb{T}:\,\sum_{j=0}^{|v|-1}T_{v|j}<t\le\sum_{j=0}^{|v|}T_{v|j}\right\}
\end{equation}
Here $|v|$ denotes the genealogical height of vertex $v$ and $v|j$ denotes the truncation up to the $j$'th generation with the convention that $v|0=0$. We note that the stochastic cascade for Navier-Stokes equations associated with the Bessel kernel $h_\text{b}$ is non-explosive, while the cascade associated with the scale-invariant kernel $h_{\text{in}}$ is explosive in dimension $d=3$ \cite{part2}.

\subsection{Thinned cascades and associated solutions}

As we noted in the introduction, Le Jan and Sznitman bypassed the stochastic explosion issue by incorporating into \eqref{eq:822191} a force term (even in the case of zero forcing) and using a somewhat ad hoc thinning procedure:
 they split the event $[T_0<t]$ into two sub-events by a Bernoulli coin flip and allowed branching to occur only on one of them. In effect, this procedure eliminates the stochastic explosion by altering the recursion \eqref{eq:822191} to\footnote{The notation $\bar{\mathbf{X}}$ should not be interpreted as complex conjugation despite the fact that the solution process is indeed complex-valued.}
\begin{equation}\label{eq:511211}\bar{\mathbf{X}}(\xi,t)=\left\{ \begin{array}{*{35}{l}}
   \chi_0(\xi) & \text{if} & {{T}_{0}}\ge t,  \\
   0 & \text{if} & {{T}_{0}}< t,\,K_0=0,  \\
   2\bar{\mathbf{X}}^{(1)}(W_1, t-{{T}_{0}}){{\odot }_{\xi }}\bar{\mathbf{X}}^{(2)}(W_2, t-{{T}_{0}}) & \text{if} & {{T}_{0}}< t,\,K_0=1  \\
\end{array} \right.\end{equation}
where $K_0$ is a Bernoulli random variable, independent of the others, with $\mathbb{P}(K_0=0)=1/2$. Similar thinning procedures are used in \cite{romito, mendes} for different differential equations. 

Visually, the thinned recursion \eqref{eq:511211} is also illustrated by \autoref{cascade} except that the tree is finite. To be specific, one needs three tree-indexed families of random variables: $\{W_v\}_{v\in\mathbb{T}}$, $\{T_v\}_{v\in\mathbb{T}}$ which satisfy properties (i)-(iii) in \autoref{introduction}, and $\{K_v\}_{v\in\mathbb{T}}$ which is an i.i.d.\ family of Bernoulli($1/2$) random variables independent of $\{W_v\}$ and $\{T_v\}$. For each $\xi \in \mathbb{R}^d\backslash\{0\}$, these families of random variables are defined on a probability space $(\bar{\Omega}, \bar{\mathcal{F}}, \bar{\mathbb{P}}_\xi)$. For the sake of simplicity of notations, we still denote by ${\mathbb{E}}_\xi$ the expectation with respect to the probability measure $\bar{\mathbb{P}}_\xi$. The recursion progresses along each path of $\mathbb{T}$, starting at the root, either branching into two at vertex $v$ if $K_v=1$ and $\sum_{j=0}^{|v|}T_{v|j}<t$, or stopping at $v$ if $K_v=0$ or $\sum_{j=0}^{|v|}T_{v|j}\ge t$. The family $\{K_v\}$ itself determines a critical Galton-Watson tree with offspring distribution $p_k=1/2$, for $k=0$ or $2$, which ensures the finiteness of the thinned recursion. The branching probability $1/2$ is not essential on this regard. In fact, one can choose $\mathbb{P}(K_0=1)=p$ for any $p\in(0,1/2]$ (the factor 2 on the the right hand side of \eqref{eq:511211} would then be replaced by $1/p$) to obtain a subcritical/critical Galton-Watson process, which gives the same effect.

If ${\mathbb{E}}_\xi|\bar{\mathbf{X}}(\xi,t)|<\infty$, 
then the \emph{thinned cascade solution} is defined as $\bar{\chi}={\mathbb{E}}_\xi\bar{\mathbf{X}}(\xi,t)$. The simple observation going back to \cite{lejan} is that if $|\chi_0|\le 1$ pointwise in $\xi$, then $|\bar{\mathbf{X}}|\le 1$ for all time, and therefore, $\chi=\mathbb{E}_\xi \bar{\mathbf{X}}$ is well-defined and solves \eqref{nFNS}.

\subsection{Incorporating explosion -- minimal cascade solutions}\label{cascadesol}

In this section, we define the minimal cascade solution to \eqref{nFNS} from the non-thinned recursion \eqref{eq:822191} and show the equivalence (under rather mild integrability conditions) between this solution and the  thinned cascade solution defined from the thinned recursion \eqref{eq:511211} by Le Jan-Sznitman and Bhattacharya et al.

Although the recursion \eqref{eq:822191} may not terminate after finitely many steps, a solution process can still be constructed on the non-explosion event $[\zeta>t]$, on which the recursion \eqref{eq:822191} stops after finitely many steps. In other words, the stochastic process 
\begin{equation}\label{lwrX}
\lwrX(\xi,t)=\mathbf{X}(\xi,t)\mathbbm{1}_{\zeta>t}
\end{equation} 
is well-defined. What connects $\lwrX$ to \eqref{nFNS} is the ``inherited'' property of the non-explosion event $\{\zeta>t\}$ when $t>T_0$: namely $\zeta=\zeta_\xi>t$ if and only if the sub-cascades starting at frequencies $W_1$ and $W_2$ have explosion times $\zeta_{_{W_1}},\zeta_{_{W_2}}>t-T_0$.  Thus, taking into account that  $\{T_0\ge t\}\subseteq\{\zeta> t\}$,  \eqref{eq:822191} yields
\begin{equation}\label{eq:822191-2}\mathbf{X}(\xi,t)\mathbbm{1}_{\zeta>t}=\left\{ \begin{array}{*{35}{l}}
   \chi_0(\xi) & \text{if} & {{T}_{0}}\ge t,  \\
   {\mathbf{X}}^{(1)}(W_1, t-{{T}_{0}})\mathbbm{1}_{\zeta_{_{W_1}}>t-T_0}{{\odot }_{\xi }}{\mathbf{X}}^{(2)}(W_2, t-{{T}_{0}})\mathbbm{1}_{\zeta_{_{W_2}}>t-T_0} & \text{if} & {{T}_{0}}< t.  \\
\end{array} \right.\end{equation}
In other words, the random variable $\lwrX(\xi,t)$, in addition to being well-defined, satisfies the recursion \eqref{eq:822191}.

We show below that, under a natural integrability assumption, the function $\chi=\mathbb{E}_\xi[\lwrX]=\mathbb{E}_\xi[\mathbf{X}\mathbbm{1}_{\zeta> t}]$ is a solution to \eqref{nFNS}. This function will be referred to as the \emph{minimal cascade solution}.

As usual, we use the Lebesgue measure on $\mathbb{R}^k$ and $\mathbb{C}^k$, $k\in\mathbb{N}$. A subset of a measurable set is said to have full measure if the complement has measure zero.
\begin{prop}\label{819191}
Let $\chi_0:\mathbb{R}^d\to\mathbb{C}^d$ be a measurable function and let $\mathbf{X}$ be a stochastic process satisfying \eqref{eq:822191}. Suppose $\mathbb{E}_\xi[|{\mathbf{X}}|\mathbbm{1}_{\zeta> t}]<\infty$ for all $(\xi,t)\in Q$, where $Q\subset \mathbb{R}^d\times[0,T)$ is a subset with full measure. Then the function $\chi(\xi,t)=\mathbb{E}_\xi[{\mathbf{X}}\mathbbm{1}_{\zeta> t}]$ is well-defined and measurable on $Q$ with the integrability property 
\begin{equation}\label{intcond}\int_{0}^{t}|\xi {{|}^{2}}{{e}^{-|\xi {{|}^{2}}s}}{\int_{{{\mathbb{R}}^{d}}}{|{{\chi }}(\eta ,s){{\odot }_{\xi }}{{\chi }}(\xi-\eta ,s)|H(\eta |\xi )d\eta ds}<\infty }\ \ \ \forall\,(\xi,t)\in Q.\end{equation}
Moreover, it satisfies {\eqref{nFNS}} everywhere in $Q$. 

\end{prop}
\begin{proof}
Note that $\lwrX(\xi,t)=\mathbf{X}(\xi,t)\mathbbm{1}_{\zeta>t}$ is well-defined and measurable on $Q\times\Omega$. Since $\mathbf{X}(\xi,t)\mathbbm{1}_{\zeta>t}\in L^1(\mathbb{P}_\xi)$, the function $\chi(\xi,t)=\int_\Omega \mathbf{X}(\xi,t,\omega)\mathbbm{1}_{\zeta>t}\mathbb{P}_\xi(d\omega)$ is measurable on $Q$ (see \cite[Thm.\ 2.39]{folland}). To show the integrability, we take the expected value of the magnitude of both side of \eqref{eq:822191-2}:
\[\infty>{{\mathbb{E}}_{\xi }}[|\lwrX|]=|{{\chi }_{0}}(\xi )|{{e}^{-|\xi {{|}^{2}}t}}+\int_{0}^{t}{|\xi {{|}^{2}}{{e}^{-|\xi {{|}^{2}}s}}\int_{{{\mathbb{R}}^{d}}}{{{\mathbb{E}}_{\xi }}\left[\left. |\lwrX^{(1)}{{\odot }_{\xi }}\lwrX^{(2)}|\,\right|\,W_1=\eta, T_0=s\right]H(\eta |\xi )d\eta ds}}\]
where $\lwrX^{(k)}=\lwrX^{(k)}(W_k,t-T_0) = {{\mathbf{X}}^{(k)}}(W_k,t-T_0)\mathbbm{1}_{\zeta_{W_k}>t-T_0}$. Using the conditional independence of $\lwrX^{(1)}$ and $\lwrX^{(2)}$ and the linearity of the expectation, we obtain that for a.e.\ $(\eta,s)\in \mathbb{R}^d\times[0,t)$,
\[
\begin{aligned}
&{{{\mathbb{E}}_{\xi }}\left[\left. |\lwrX^{(1)}{{\odot }_{\xi }}\lwrX^{(2)}|\,\right|\,W_1=\eta, T_0=s\right]}
\ge \left|{{{\mathbb{E}}_{\xi }}\left[\left. \lwrX^{(1)}{{\odot }_{\xi }}\lwrX^{(2)}\,\right|\,W_1=\eta, T_0=s\right]}\right|
\\
&=
\left|{{\mathbb{E}}_{{\eta}}}{{\lwrX}^{(1)}}(\eta,t-s)\spr{\xi} {{\mathbb{E}}_{{\xi-\eta}}}{{\lwrX}^{(2)}}(\xi-\eta,s)\right|
=|\chi(\eta,t-s) {{\odot }_{\xi }}\chi(\xi-\eta,t-s) |,
\end{aligned}\]
and thus (\ref{intcond}) follows.

To show $\chi=\mathbb{E}_\xi \lwrX$ satisfies {\eqref{nFNS}}, we take expectation of both sides of \eqref{eq:822191-2} (with $(\xi,t)\in Q$):
\[
\chi(\xi,t)={{\mathbb{E}}_{\xi }}[\lwrX(\xi,t)]={{\chi }_{0}}(\xi ){{e}^{-|\xi {{|}^{2}}t}}+\int_{0}^{t}{|\xi {{|}^{2}}{{e}^{-|\xi {{|}^{2}}s}}\int_{{{\mathbb{R}}^{d}}}{{{\mathbb{E}}_{\xi }}\left[\left. \lwrX^{(1)}{{\odot }_{\xi }}\lwrX^{(2)}\,\right|\,W_1=\eta, T_0=s\right]H(\eta |\xi )d\eta ds}}\,.
\]
Thus, {\eqref{nFNS}} holds in $Q$. 
\end{proof}

\begin{rem}
Note that under conditions of \autoref{819191}, the equation {\eqref{nFNS}} automatically holds at $t=0$ with 
$\mathbb{E}[\mathbf{X}(\xi,0)]=\mathbb{E}[\lwrX(\xi,0)]=\chi_0(\xi)$.
\end{rem}
%

At this point, given the initial data, one can construct solutions to \eqref{nFNS} in two different ways: by either \eqref{eq:822191} or \eqref{eq:511211}. We refer to the two solutions  as \emph{cascade solutions} due to the nature of the constructions. A natural question is whether they coincide with each other. Under mild assumptions on the integrability, it will be shown that they indeed coincide with each other and are given as the pointwise limit of the Picard's iteration
\begin{equation}\label{eq:71211}{{\chi}^{(0)}}=0,\ \ \ {{\chi}^{(n+1)}}={U}(\chi_0)+\mathfrak{B}({{\chi}^{(n)}},{{\chi}^{(n)}})
\end{equation} 
where ${U}(w)=e^{-t|\xi|^2}w$ and 
\begin{equation}\label{eq:71212}\mathfrak{B}(f,g)=\int_{0}^{t}{{{e}^{-s|\xi {{|}^{2}}}}|\xi {{|}^{2}}\int_{{{\mathbb{R}}^{d}}}{f (\eta ,t-s)\odot_{\xi}g (\xi -\eta ,t-s)H(\eta |\xi )d\eta ds }}.
\end{equation}

\begin{thm}\label{531211}
Let $\chi_0:\mathbb{R}^d\to\mathbb{C}^d$ be a measurable function and $T\in[0,\infty]$. Let $\mathbf{X}$ and $\bar{\mathbf{X}}$ be the stochastic processes defined by \eqref{eq:822191} and \eqref{eq:511211} respectively. Let $\lwrX(\xi,t)=\mathbf{X}(\xi,t)\mathbbm{1}_{\zeta>t}$. We have the following statements.
\begin{enumerate}
\item[\textup{(a)}] If $\mathbb{E}_\xi|\bar{\mathbf{X}}|<\infty$ a.e.\ on $\mathbb{R}^d\times[0,T)$ then the thinned cascade solution of {\eqref{nFNS}} is equal to the pointwise limit of the Picard's iteration \eqref{eq:71211}.
\item[\textup{(b)}] If $\mathbb{E}_\xi|\bar{\mathbf{X}}|<\infty$ a.e.\ on $\mathbb{R}^d\times[0,T)$ then the minimal cascade solution of {\eqref{nFNS}} is equal to the pointwise limit of the Picard's iteration \eqref{eq:71211}.
\item[\textup{(c)}] $\mathbb{E}_\xi |\uprX| =\mathbb{E}_\xi|\lwrX|\in[0,\infty]$ a.e.\ in $\mathbb{R}^d\times[0,\infty)$,
and therefore,  $\mathbb{E}_\xi \uprX =\mathbb{E}_\xi \lwrX$ a.e.\ in $\mathbb{R}^d\times[0,T)$ whenever one of the expectations is finite.
\end{enumerate}
\end{thm}
\begin{proof}
The proof of Part (a) can be found in \cite[Prop.\ 4.3]{rabi}. For the sake of completeness, we include a brief version of Bhattacharya et al.'s proof as follows. First, one truncates the recursion \eqref{eq:511211} by introducing $\bar{\mathbf{X}}_n=\bar{\mathbf{X}}\mathbbm{1}_{G(\xi,t)<n}$ where
\begin{eqnarray*}G(\xi,t)=&\sup&\left\{|v|:\,v\in\mathbb{T},\prod\limits_{j=0}^{|v|-1}{{{K}_{v|j}}}=1,\,\sum_{j=0}^{|v|-1}T_{v|j}<t,~\text{and}\right.\\
&&\left.\text{either}~K_v=0~\text{or}~\left(K_v=1,\,\sum_{j=0}^{|v|}T_{v|j}\ge t\right)\right\}.\end{eqnarray*}
Intuitively, ${G}(\xi,t)$ is the genealogical height of the (finite) tree in \autoref{cascade}. Thanks to the inherited property $\mathbbm{1}_{{G}(\xi,t)<n}\mathbbm{1}_{T_0<t}=\mathbbm{1}_{{G}({W_1},t-T_0)<n-1}\mathbbm{1}_{{G}({W_2},t-T_0)<n-1}\mathbbm{1}_{T_0<t}$, the sequence $\bar{\mathbf{X}}_n$ satisfies a truncated version of \eqref{eq:511211}:
\begin{equation}\label{526212}\begin{array}{*{35}{l}}
   {\bar{\mathbf{X}}_{0}}= 0,\ {\bar{\mathbf{X}}_{n}}(\xi ,t)=\left\{ \begin{array}{*{35}{l}}
   {{\chi }_{0}}(\xi ) & \text{if} & {{T}_{0}}\ge t,  \\
      0 & \text{if} & {{T}_{0}}< t,\,K_0=0,  \\
   2\bar{\mathbf{X}}_{n-1}^{(1)}({{W}_{1}},t-{{T}_{0}})\odot_\xi\bar{\mathbf{X}}_{n-1}^{(2)}(W_2,t-{{T}_{0}}) & \text{if} & {{T}_{0}}< t,\,K_0=1.  \\
\end{array} \right.  \\
\end{array}\end{equation}
By conditioning on $T_0$, $K_0$, $W_1$, $W_2$, we obtain
\begin{equation*}\mathbb{E}_{\xi}\bar{\mathbf{X}}_n={{e}^{-t|\xi {{|}^{2}}}}{{\chi }_{0}}(\xi)+
\int_{0}^{t}{{{e}^{-s|\xi {{|}^{2}}}}|\xi {{|}^{2}}\int_{{{\mathbb{R}}^{d}}}{\mathbb{E}_{\eta}\bar{\mathbf{X}}_{n-1}(\eta,t-s)\odot_\xi\mathbb{E}_{\xi -\eta}\bar{\mathbf{X}}_{n-1}(\xi -\eta,t-s)H(\eta |\xi )d\eta ds. }}\end{equation*}
Since the thinned tree is non-explosive,  $\bar{\mathbf{X}}_n\to \bar{\mathbf{X}}$ a.s. Note that we also have $|\bar{\mathbf{X}}_n|\le|\bar{\mathbf{X}}|$, and therefore, by Lebesgue's Dominated Convergence Theorem, the sequence $\bar{\chi}^{(n)}:=\mathbb{E}_\xi\bar{\mathbf{X}}_n$ converges pointwise to $\bar\chi:=\mathbb{E}_\xi\bar{\mathbf{X}}$ and satisfies \eqref{eq:71211}.

For Part (b), one truncates the recursion \eqref{eq:822191} by introducing ${\lwrX}_n={\mathbf{X}}\mathbbm{1}_{\zeta_{\xi,n}> t}$ where $\zeta_{\xi,n}=\min_{|v|=n-1}\sum_{j=0}^{n-1}T_{v|j}$ with the convention that $\zeta_{\xi,0}=0$. Intuitively, $\{\zeta_{\xi,n}>t\}$ is the event that every path of the cascade in \autoref{cascade} crosses the horizon $t$ after less than $n$ times of branching. 
Note that in the explosive event $[\zeta<t]$, $\lwrX_n=0=\lwrX$ for all $n$. In the event of no explosion, i.e. $\zeta>t$, the corresponding cascade tree is finite and therefore $\lwrX_n=\lwrX$ for big enough $n$ (depending on the outcome $\omega$). Thus,  a.s.\ $\lim\lwrX_n=\lwrX=\mathbf{X}\mathbbm{1}_{\zeta>t}$ and $|\lwrX_n|\le|\mathbf{X}|\mathbbm{1}_{\zeta>t}$. By Lebesgue's Dominated Convergence Theorem, ${\chi}:=\mathbb{E}_\xi[{\mathbf{X}}\mathbbm{1}_{\zeta> t}]=\lim\mathbb{E}_\xi\lwrX_n$. By the inherited property $\mathbbm{1}_{\zeta_{\xi,n}>t}\mathbbm{1}_{T_0<t}=\mathbbm{1}_{\zeta_{W_1,n-1}>t-T_0}\mathbbm{1}_{\zeta_{W_2,n-1}>t-T_0}\mathbbm{1}_{T_0<t}$, the sequence ${\mathbf{X}}_n$ satisfies a truncated version of \eqref{eq:822191}:
\begin{equation}\label{526211}\begin{array}{*{35}{l}}
   {\lwrX_{0}}(\xi ,t)= 0,\ \ {\lwrX_{n}}(\xi ,t)=\left\{ \begin{array}{*{35}{l}}
   {{\chi }_{0}}(\xi ) & \text{if} & {{T}_{0}}\ge t,  \\
   \lwrX_{n-1}^{(1)}({{W}_{1}},t-{{T}_{0}})\odot_\xi\lwrX_{n-1}^{(2)}(W_2,t-{{T}_{0}}) & \text{if} & {{T}_{0}}< t.  \\
\end{array} \right.  \\
\end{array}\end{equation}
By conditioning on $T_0$, $W_1$, $W_2$, one arrives at
\begin{equation*}\mathbb{E}_{\xi}{\lwrX}_n={{e}^{-t|\xi {{|}^{2}}}}{{\chi }_{0}}(\xi)+\int_{0}^{t}{{{e}^{-s|\xi {{|}^{2}}}}|\xi {{|}^{2}}\int_{{{\mathbb{R}}^{d}}}{\mathbb{E}_{\eta}{\lwrX}_{n-1}(\eta ,t-s)\odot_\xi\mathbb{E}_{\xi -\eta}{\lwrX}_{n-1}(\xi-\eta,t-s)H(\eta |\xi )d\eta ds. }}\end{equation*}
Thus, $\chi^{(n)}=\mathbb{E}_\xi\lwrX_n$ satisfies
\eqref{eq:71211}
as claimed and converges pointwise to $\chi$. 

For Part (c), observe that on the explosion event $\{\zeta<t\}$, $\uprX=\lwrX=0$. On the non-explosion event $\{\zeta>t\}$, let $N=\mathrm{card}\,V(\xi,t)<\infty$. We either have $\uprX=2^{N-1}\lwrX$ if all the coin tosses $K_w=1$ for all ancestors $w$ of the vertices in $V(\xi,t)$, or otherwise, $\uprX=0$. Because the number of such ancestors is exactly $N-1$ and that the coin tosses are i.i.d.\ Bernoulli random variables, conditioning on the number of leaves we obtain:
\[
\begin{aligned}
\mathbb{E}_\xi|\uprX|&=
\mathbb{E}_\xi[|\uprX|\mathbbm{1}_{\zeta>t}]=\sum\limits_{n=1}^{\infty} \,\mathbb{E}_\xi\left[\left. 2^{n-1}|\lwrX|\,\right|\, N=n\right]
\left(\frac{1}{2}\right)^{n-1}\mathbb{P}_{\xi}(N=n)=\\
&=\sum\limits_{n=1}^{\infty} \,\mathbb{E}_\xi\left[\left. |\lwrX|\,\right|\, N=n\right]
\mathbb{P}_{\xi}(N=n)
=\mathbb{E}_\xi|\lwrX|.
\end{aligned}
\]  
As a consequence of Parts (a) and (b), if $\mathbb{E}_\xi |\lwrX|<\infty$ or $\mathbb{E}_\xi |\uprX|<\infty$ a.e. on $\mathbb{R}^d\times[0,T)$ for some $T>0$, then 
$\mathbb{E}_\xi \lwrX= \mathbb{E}_\xi \uprX$ a.e. on $\mathbb{R}^d\times[0,T)$. 
\end{proof}

\begin{rem} While, when defined, the minimal solution $\mathbb{E}_\xi \lwrX$ coincides with the thinned solution $\mathbb{E}_\xi \uprX$, the presence of the factor $2$ in front of the $\spr{}$ in \eqref{eq:511211} translates into a factor of $2^{N-1}$ difference between $\lwrX$ and $\uprX$ in the proof above. 
In practice, this inflation restricts the size of $\chi_0$ for which integrability of the thinned solution can be established, effectively narrowing the set of initial data for which well-posedness can be shown. To improve the estimates on the initial data, it is more advantageous to use the minimal solution.
\end{rem} 

\section{Majorizing principle and symmetry}\label{reduction}
We begin this section by establishing a relation between the explosion time defined in \eqref{explosiontime} and a toy model of the Navier-Stokes equations introduced by Montgomery-Smith in \cite{smith}.    We first note that the probability of non-explosion $\rho(\xi,t) = \mathbb{E}_\xi[\mathbbm{1}_{\zeta> t}]$ is a solution for an integral equation that has the constant function $1$ also as a solution. Recall that in the construction of the branching process as defined in (i), (ii) and (iii) of \autoref{introduction}, the distribution of the wave vector $W_{v1}$ given $W_v=\xi$ is $H(\eta|\xi)=h(\eta)h(\xi-\eta)/(|\xi|h(\xi))$ where $h$ is a positive function satisfying $h*h(\xi) = |\xi| h(\xi)$.  We have
\begin{prop}
\label{MSviaExplosion}
Assume that $\{T_v\}_{v\in\mathbb{T}} and \{W_v\}_{v\in\mathbb{T}}$ are defined as in (i),(ii) and (iii) of \autoref{introduction}. Let $\rho(\xi,t) = \mathbb{E}_\xi[\mathbbm{1}_{\zeta> t}]$.  Then
\begin{equation}
\label{explosionequation}
\rho(\xi,t) = e^{-|\xi|^2 t} + \int_0^t |\xi|^2 e^{-|\xi|^2 s} \int_{\mathbb{R}^d} \rho(\eta,t-s)\rho(\xi-\eta,t-s) H(\eta|\xi) d\eta ds
\end{equation}
and consequently, $\rho$ is continuous in time.
\end{prop}
\begin{proof}
This follows similar lines as the proof of \autoref{819191}.  Note that $Y(\xi,t)=\mathbbm{1}_{\zeta_\xi>t}$ satisfies
\begin{equation}\label{pr_no_expl}
Y(\xi,t)=\left\{\begin{array}{ll}
1, & T_0\ge t,\\
Y^{(1)}(\eta, t-T_0)Y^{(2)}(\xi-\eta,t-T_0),\quad & T_0<t,
\end{array}
\right.
\end{equation}
where $Y^{(1)}$ and $Y^{(2)}$ are two conditionally independent copies on $Y$.
Indeed, this follows from the inclusion $[T_0\ge t]\subset[\zeta> t]$ and from inherited nature of the non-explosion event: 
${\mathbbm{1}_{\zeta> t}}\mathbbm{1}_{T_0<t}=\mathbbm{1}_{T_0<t}\left(\mathbbm{1}_{{{\zeta}_{{{W}_{1}}}}> t-T_0}+\mathbbm{1}_{{{\zeta}_{{{W}_{2}}}}> t-T_0}\right)$.
Since $|Y|\le1$, we can take expectation in \eqref{pr_no_expl} to obtain that $\rho=\mathbb{E}Y$ satisfies \eqref{explosionequation}.
The continuity of $\rho$ follows immediately from \eqref{explosionequation} since $\int_{\mathbb{R}^3} \rho(\eta,t-s)\rho(\xi-\eta,t-s) H(\eta|\xi) d\eta\leq 1.$
\end{proof}
Note that the constant function $1$ is a trivial solution of \eqref{explosionequation}.  Consequently, determining whether stochastic explosion occurs (i.e.\ $\rho(\xi,t) < 1$ for some $\xi,t$) is equivalent to determining whether \eqref{explosionequation} has other solutions besides the trivial solution.  

In 2001, Montgomery-Smith introduced a toy model of the Navier-Stokes equations, called the ``cheap Navier-Stokes equation'', which has the same scaling symmetry as the Navier-Stokes equations and for which he constructed a finite-time blowup solution \cite{smith}:
\begin{equation}\label{MS}\tag{MS}\left\{ \begin{array}{*{35}{rclcl}}
   {{\partial }_{t}}u-\Delta u & = & \sqrt{-\Delta }({{u}^{2}}) & \text{in} & {{\mathbb{R}}^{d}}\times (0,\infty ),  \\
   u(\cdot ,0) & = & {{u}_{0}} & \text{in} & {{\mathbb{R}}^{d}} . \\
\end{array} \right.\end{equation}
Here $u$ and $u_0$ are scalar valued functions. 
The Fourier transform of \eqref{MS} can be obtained by simply replacing the circle-dot product in \eqref{FNS} with the regular product (of numbers):
\begin{equation}\label{FMS}
\tag{FMS} {v}=U({v}_0)+\widetilde{B}({v},{v})
\end{equation}
where $v=\mathscr{F}\{u\}$, $v_0=\mathscr{F}\{u_0\}$, $U(v_0)=e^{-t|\xi|^2}v_0$ and
\begin{equation}\label{eq:71213}
\widetilde{B}(f,g)=c_0\int_{0}^{t}{{{e}^{-s|\xi {{|}^{2}}}}|\xi |\int_{{{\mathbb{R}}^{d}}}{f(\eta ,t-s)g(\xi -\eta ,t-s)d\eta ds}}.
\end{equation}
Note that $h(\xi)$ is a steady state solution of \eqref{FMS}. With a similar normalization $\psi=c_0{{v}}/{h}$ and $\psi_0=c_0{{v_0}}/{h}$, where $h$ is a standard majorizing kernel,
\eqref{FMS} becomes
\begin{equation}\label{nFMS}\tag{nFMS}\psi (\xi ,t)={{e}^{-t|\xi {{|}^{2}}}}{{\psi }_{0}}(\xi)+\int_{0}^{t}{{{e}^{-s|\xi {{|}^{2}}}}|\xi {{|}^{2}}\int_{{{\mathbb{R}}^{d}}}{\psi (\eta ,t-s)\psi (\xi -\eta ,t-s)H(\eta |\xi )d\eta ds }}.\end{equation}
We will sometimes write \eqref{nFMS}, \eqref{nFNS},... as \eqref{nFMS}$_{\psi_0}$, \eqref{nFNS}$_{\chi_0}$,... to emphasize the initial data. The connection with the stochastic explosion is now clear
: since both $1$ and, by \autoref{MSviaExplosion}, the probability of non-explosion $\rho(\xi,t)=\mathbb{E}_\xi[\mathbbm{1}_{\zeta>t}]$ are solutions of the equation $\eqref{nFMS}_1$, stochastic explosion does not occur if and only if there is a unique solution to $\eqref{nFMS}_1$ satisfying $0\le\chi(\xi,t)\le1$ for all $\xi,t$.  In this sense, the equation introduced by Montgomery-Smith plays a central role in the study of stochastic explosion for the Navier Stokes equations.  In \autoref{nonuniqueness},  we illustrate the nonuniqueness of solutions in the case $h(\xi)=1/(\pi^3|\xi|^2)$ (see \cite{part2}) and show how the cascade solution can be used to show blow up of solutions of \eqref{FMS}.  This current section is devoted to strengthening the comparison principle between \eqref{nFMS} and \eqref{nFNS} first established in \cite{lejan} by exploiting the symmetrized $\odot$ product and some symmetry properties of the solution process. 

With this purpose in mind, we only consider initial data $\psi_0(\xi)\ge 0$ for all $\xi$. In particular, when comparing \eqref{nFNS} with \eqref{nFMS}, we will be interested in the case $\psi_0=|\chi_0|/2$. One can associate the following stochastic processes with \eqref{nFMS}$_{\psi_0}$.
\begin{itemize}
\item Without thinning:
\begin{equation}\label{eq:819191}\mathbf{Y}(\xi,t)=\left\{ \begin{array}{*{35}{l}}
   \psi_0(\xi) & \text{if} & {{T}_{0}}\ge t,  \\
   {\mathbf{Y}}^{(1)}(W_1, t-{{T}_{0}}){\mathbf{Y}}^{(2)}(W_2, t-{{T}_{0}}) & \text{if} & {{T}_{0}}< t.  \\
\end{array} \right.\end{equation}
\item With thinning: 
\begin{equation}\label{eq:522211}\bar{\mathbf{Y}}(\xi,t)=\left\{ \begin{array}{*{35}{l}}
   \psi_0(\xi) & \text{if} & {{T}_{0}}\ge t,  \\
   0 & \text{if} & {{T}_{0}}< t,\,K_0=0,  \\
   2\bar{\mathbf{Y}}^{(1)}(W_1, t-{{T}_{0}})\bar{\mathbf{Y}}^{(2)}(W_2, t-{{T}_{0}}) & \text{if} & {{T}_{0}}< t,\,K_0=1.  \\
\end{array} \right.\end{equation}
\end{itemize}
We will write $\mathbf{Y}_{\psi_0}$ and $\bar{\mathbf{Y}}_{\psi_0}$ if reference to the initial data is needed. Observe that \eqref{nFMS}$_{\psi_0}$ and \eqref{nFNS}$_{\chi_0}$ have the same underlying stochastic structure, namely the couple $(\{W_v\},\{T_v\})$ in case of non-thinning and the triple $(\{W_v\},\{T_v\},\{K_v\})$ in case of thinning. The minimal solution process  $\lwrY={\mathbf{Y}}\mathbbm{1}_{\zeta> t}$ and the thinned solution process $\bar{\mathbf{Y}}$ are well-defined. 
One can likewise obtain cascade solutions to \eqref{nFMS}$_{\psi_0}$ from the recursions \eqref{eq:819191} and \eqref{eq:522211}, referred to as minimal cascade solution $\psi=\mathbb{E}_\xi\lwrY= \mathbb{E}_\xi[{\mathbf{Y}}\mathbbm{1}_{\zeta> t}]$ and thinned cascade solution $\bar{\psi}=\mathbb{E}_\xi\bar{\mathbf{Y}}$ in the analogy with \eqref{nFNS}$_{\chi_0}$. Similarly to \eqref{nFNS}, it can be shown
that these solutions are equal to each other and given by the pointwise limit of the Picard's iteration 
\begin{equation}\label{eq:71217}
{{\psi}^{(0)}}=0,\ \ \ {{\psi}^{(n+1)}}=U(\psi_0)+\widetilde{\mathfrak{B}}({{\psi}^{(n)}},{{\psi}^{(n)}}),
\end{equation}
where
\begin{equation}\label{eq:71216}
\widetilde{\mathfrak{B}}(f,g)=c_0\int_{0}^{t}{{{e}^{-s|\xi {{|}^{2}}}}|\xi {{|}^{2}}\int_{{{\mathbb{R}}^{d}}}{f (\eta ,t-s)g (\xi -\eta ,t-s)H(\eta |\xi )d\eta ds }}.
\end{equation}
The simplified product allows simple methods to control the solution of \eqref{nFMS}$_{\psi_0}$ which are not available for \eqref{nFNS}$_{\chi_0}$, for example the comparison principle. The inequality $|a\odot_\xi b|\le \frac{1}{2}|a||b|$ in turn provides a natural way to bound the solution of \eqref{nFNS}$_{\chi_0}$ by the solution of \eqref{nFMS}$_{|\chi_0|/2}$. Le Jan and Sznitman showed that $\bar{\psi}$ is a solution of \eqref{nFMS}$_{|\chi_0|}$ and that $|\bar{\chi}|\le \bar{\psi}$, where $\bar{\chi}=\mathbb{E}_\xi\bar{\mathbf{X}}$ is the thinned cascade solution of \eqref{nFNS}$_{\chi_0}$. 
 This is known as a majorizing principle.
 Note the absence of the factor $1/2$ in the initial data comparison: $\psi_0=|\chi_0|$ in Le Jan-Sznitmann's approach and
$\psi_0=|\chi_0|/2$ in our approach (shown below). This improvement is due to our use of the symmetrized product $\spr{}$ compared to the non-symmetrized product $\otimes$ used in \cite{lejan}. Such an improvement leads to well-posedness results of \eqref{FNS} for larger initial data. 


%

\begin{prop}\label{69211}
Let $\chi_0:\mathbb{R}^d\to\mathbb{C}^d$ be a measurable function and let $\psi_0=|\chi_0|/2$. We have the following statements.
\begin{enumerate}[(a)]
\item The function $\psi(\xi,t)=\mathbb{E}_\xi\lwrY=\mathbb{E}_\xi[{\mathbf{Y}}\mathbbm{1}_{\zeta> t}]\in [0,\infty]$ satisfies {\eqref{nFMS}}$_{\psi_0}$ \textup{(}with the convention $0\cdot\infty=0$\textup{)}. Moreover, $\psi=\bar{\psi}$ and is given as a pointwise limit of the Picard's iteration \eqref{eq:71217}.
\item For each $\xi\neq 0$, $e^{t|\xi|^2}\psi$ is a nondecreasing function of $t$.
\item \textup{(Majorizing principle)}
Suppose that for some $T\in(0,\infty]$, $\psi<\infty$ a.e.\ on $\mathbb{R}^d\times(0,T)$. Then there exists a subset $D\subset \mathbb{R}^d$ with full measure such that both $\mathbb{E}_\xi|\bar{\mathbf{X}}|$ and $\mathbb{E}_\xi[|{\mathbf{X}}|\mathbbm{1}_{\zeta>t}]$ are finite on $D_T=D\times[0,T)$. 
Moreover, the cascade solutions to {\eqref{nFNS}}$_{\chi_0}$ are well-defined on $D_T$, equal to each other \textup{(}denoted by $\chi$\textup{)}, given by the pointwise limit of the Picard's iteration \eqref{eq:71211}, and satisfy $|\chi|\le 2\psi$ on $D_T$.
\item \textup{(Comparison principle of nFMS)} 
Let $\phi_0,\gamma_0:\mathbb{R}^d\to[0,\infty]$ be measurable functions, and let $\phi$ and $\gamma$ be the minimal cascade solutions to {\eqref{nFMS}}$_{\phi_0}$ and {\eqref{nFMS}}$_{\gamma_0}$, respectively. If $\phi_0\le\gamma_0$ on $\mathbb{R}^d$ then $\phi\le\gamma$ on $\mathbb{R}^d\times(0,\infty)$.
\end{enumerate}
\end{prop}
\begin{proof}
The proof of Part (a) follows the same lines as the proof of \autoref{819191} and \autoref{531211} and will be omitted. To show Part (b), we multiply both sides of \eqref{nFMS}$_{\psi_0}$ by $e^{t|\xi|^2}$ and get \begin{equation}\label{529211}e^{t|\xi|^2}\psi-\psi_0(\xi)=\int_0^t|\xi|^2e^{s|\xi|^2}f(\xi,s)ds\end{equation} where $f(\xi,s)=\int\psi(\eta,s)\psi(\xi-\eta,s)H(\eta|\xi)d\eta\ge 0$. The monotonicity of $e^{t|\xi|^2}\psi$ immediately follows. 

For Part (c), we first show the existence of a subset $D\subset\mathbb{R}^d$ with full measure such that $\psi$ is finite on $D_T$. Let $A=\{(\xi,t)\in\mathbb{R}^d\times(0,T): \psi(\xi,t)=\infty\}$. For each $t\in[0,T)$, let $A_t=\{\xi\in\mathbb{R}^d: \psi(\xi,t)=\infty\}$. By Fubini's theorem, $0=m(A)=\int_0^T m(A_t)dt$, where $m$ denotes the Lebesgue measure. This implies $A_t$ is null set for a.e.\ $t\in[0,T)$. By Part (b), $A_t\subset A_s$ whenever $t<s$. Hence, $\cup_{t\in(0,T)}A_t$ is in fact a countable union of null sets. One can choose $D=\mathbb{R}^d\backslash\cup_{t\in[0,T)}A_t$. 
Next, as noted in \autoref{setup}, it can be observed from \eqref{eq:822191} that on the event $[\zeta> t]$, $\mathbf{X}$ is a circle-dot product of $\chi_0(W_v)$, $v\in V(\xi,t)$, in a suitable order, where $V(\xi,t)$, is given by given by \eqref{Vxit}, is the set of leaves of the cascade tree that reached time $t$.  Thanks to the inequality $|a\odot_\xi b|\le\frac{1}{2} |a||b|$ (see \eqref{sprod-bd}),  one has $|\mathbf{X}|\le 2\mathbf{Y}$. Let $(\xi,t)\in D\times[0,T)$, and take the expected values of both sides and use Part (b):
\[\mathbb{E}_\xi[|\mathbf{X}|\mathbbm{1}_{\zeta>t}]\le {2}\mathbb{E}_\xi[\mathbf{Y}\mathbbm{1}_{\zeta>t}]
<\infty\]
for any $(\xi,t)\in D_T$ and $s\in(t,T)$. By \autoref{819191}, the minimal cascade solution to \eqref{nFNS}$_{\chi_0}$ is well-defined on $D_T$. 

Part (d) follows by noting that on the event $[\zeta> t]$, $\mathbf{Y}$ is a finite product of values of $|\chi_0|$:
\begin{equation}\label{eq:49191}
\mathbf{Y}(\xi ,t,\omega)=\prod\limits_{{v}\in V(\xi ,t)}|{{{\chi }_{0}}({{W}_{{v}}})}|\end{equation}
where $V(\xi,t)$ is the random set given by \eqref{Vxit}.
Thus, an ordering in the initial data implies an ordering in the random variables, resulting in an ordering of the solutions of \eqref{nFMS}.
\end{proof}
According to \autoref{69211}, the finiteness of the cascade solution to \eqref{nFMS} with initial data larger than or equal to $|\chi_0|/2$ leads to the existence of cascade solutions to \eqref{nFNS} with initial data $\chi_0$. This result can be generalized by observing a natural algebraic symmetry of the multiplicative structure of \eqref{eq:819191}:
\begin{equation}\label{symm}|\chi_0|\to f(|\chi_0|),\ \ \ \mathbf{Y}\to f(\mathbf{Y}), 
\quad\mbox{or in other words, $\mathbf{Y}_{f(|\chi_0|)}=f(\mathbf{Y}_{\chi_0})$},
\end{equation}
where $f$ is a multiplicative scalar function, i.e.\ $f(ab)=f(a)f(b)$. 
Relaxing the property of $f$ leads to a generalized majorizing principle. The precise statement requires the introduction of a class of functions
\begin{equation}\label{eq:73211}
S=\{f:[0,\infty)\to[0,\infty)\text{\ submultiplicative,\ convex,}\ f(0)=0,\ f\not\equiv 0\}.
\end{equation}
Here submulplicativity means $f(ab)\le f(a)f(b)$ for all $a,b\in[0,\infty)$.
\begin{rem}
It is worth mentioning here some properties of the set $S$. 
\begin{enumerate}[(i)]
\item The power functions $x^\alpha$ for $\alpha\ge 1$ obviously belong to $S$. Less obvious elements of $S$ include functions of the form $x^\alpha\ln(x^\beta+\gamma)$ for $\alpha\ge 1$, $\beta\ge 0$, $\gamma\ge e^2$.\vspace{-4pt}
\item Each $f\in S$ is continuous on $[0,\infty)$. The continuity on $(0,\infty)$ is due to the convexity. The continuity at 0 is because $f(x)=f(x1+(1-x)0)\le xf(1)$ for all $x\in(0,1)$.\vspace{-4pt}
\item Each $f\in S$ is positive, increasing on $(0,\infty)$ and $\lim_{x\to\infty}f(x)=\infty$. The positivity derives from the submultiplicative property. The monotonicity is because $f(x)=f(\frac{x}{y}y+(1-\frac{x}{y})0)\le \frac{x}{y}f(y)<f(y)$ for $0\le x<y$. By the convexity, $f$ grows at least as fast as a linear function as $x\to \infty$. \vspace{-4pt}
\item If $f\in S$ and $c\ge 1$ then $cf\in S$. \vspace{-4pt}
\item If $f_1,f_2\in S$ then $f_1+f_2$, $f_1f_2$, $f_1\circ f_2\in S$. In other words, $S$ is closed under the addition, multiplication and composition.\vspace{-4pt}
\item If $\{f_i\}_{i\in I}\subset S$ and $f(x)=\sup_if_i(x)$ is finite everywhere then $f\in S$.\vspace{-4pt}
\item If $f\in S$, the function $g(t)=\ln f(e^t)$ is subadditive on $\mathbb{R}$. Basic properties of subadditive functions can be found in \cite[Ch.\ 7]{hille} and \cite[Ch.\ 16]{kuczma}. 
\end{enumerate}
\end{rem}
\begin{thm}[Generalized Majorizing Principle]\label{614211}
Let $\chi_0:\mathbb{R}^d \to\mathbb{C}^d$ be a measurable function and $f\in S$. Let $\phi$ be the minimal cascade solution to {\eqref{nFMS}}$_{f(|\chi_0|/2)}$. Suppose that for some $T\in(0,\infty]$, $\phi<\infty$ a.e.\ on $\mathbb{R}^d\times[0,T)$. 
Then there exists a subset $D\subset \mathbb{R}^d$ of full measure such that {\eqref{nFNS}}$_{\chi_0}$ has a minimal cascade solution $\chi$ on $D\times[0,T)$ satisfying $|\chi|\le 2f^{-1}(\phi)$. 
\end{thm}

\begin{proof}
Let $\lwrY_{\phi_0}=\mathbf{Y}_
{\phi_0}\mathbbm{1}_{\zeta>t}$ be the minimal stochastic process defined by \eqref{eq:819191} corresponding to the initial data $\phi_0=f(|\chi_0|/2)$ and $\lwrY_{\psi_0}=\mathbf{Y}_{\psi_0}\mathbbm{1}_{\zeta>t}$ be the minimal stochastic process defined by \eqref{eq:819191} corresponding to the initial data $\psi_0=|\chi_0|/2$. 
 On the explosion event $[\zeta<t]$, $\lwrY_{\phi_0}=\lwrY_{\psi_0}=0$. On the non-explosion  event $[\zeta>t]$,
 we apply $f$ to both sides of \eqref{eq:49191} and use the submultiplicativity of $f$ to obtain $f(\lwrY_{\psi_0})\le\prod_{{v}\in V(\xi ,t)}
 f(|{{{\chi }_{0}}({{W}_{{v}}})}|/2)=\lwrY_{\phi_0}$, 
where $V(\xi,t)$ as in the proof of \autoref{69211}, part (d). Taking the expectation of both sides on the non-explosion event and using the fact that $f(0)=0$, one gets 
\[\mathbb{E}_\xi[f(\lwrY_{\psi_0})]\le {{\mathbb{E}}_{\xi }}[\lwrY_{\phi_0}]=\phi.\]
By Jensen's inequality, $f(\psi)\le\phi$ where $\psi=\mathbb{E}_\xi[\mathbf{Y}\mathbbm{1}_{\zeta>t}]$ is the cascade solution of \eqref{nFMS}$_{|\chi_0|/2}$. Thus, $\psi\le f^{-1}(\phi)<\infty$ a.e.\ on $\mathbb{R}^d\times(0,T)$. The proof is then completed by \autoref{69211} (c).
\end{proof}
An analytic proof of \autoref{614211} will be given in the appendix. Next, we have another variation of the majorzing principle thanks to H\"{o}lder's inequality.
\begin{prop}\label{94191}
Let $k\ge 1$ and $\alpha_1,\ldots,\alpha_k\in[0,1]$ be such that $\alpha_1+\ldots+{\alpha_k}\le 1$. Consider measurable functions $\chi_0:\mathbb{R}^d \to\mathbb{C}^d$ and $\chi_{01},\ldots,\chi_{0k}:\mathbb{R}^d \to[0,\infty)$ satisfying $|\chi_0|\le 2\chi_{01}^{\alpha_1}\ldots \chi_{0k}^{\alpha_k}$. Let $\chi_j$ be the cascade solution to {\eqref{nFMS}}$_{\chi_{0j}}$ for $j=1,\ldots k$. Suppose that for some $T\in(0,\infty]$, $\chi_1,\ldots, \chi_k<\infty$ a.e.\ on $\mathbb{R}^d\times[0,T)$. Then for some subset $D\subset \mathbb{R}^d$ with full measure, {\eqref{nFNS}}$_{\chi_0}$ has a minimal cascade solution $\chi$ on $D\times[0,T)$ and $|\chi|\le 2\chi_{1}^{\alpha_1}\ldots\chi_{k}^{\alpha_k}$.
\end{prop}
\begin{proof}
Let $\mathbf{Y}$, $\mathbf{X}_1$,\ldots, $\mathbf{X}_k$ be the stochastic processes defined in the fashion of \eqref{eq:819191} with initial data $|\chi_0|/2$, $\chi_{01}$, \ldots, $\chi_{0k}$, respectively. By  the representation \eqref{eq:49191}, $\mathbf{Y}\le\mathbf{X}_1^{\alpha_1}\ldots \mathbf{X}_k^{\alpha_k}$. By H\"{o}lder's inequality and the fact that $\mathbb{E}[1]=1$,
\[{{\mathbb{E}}_{\xi }}[\mathbf{Y}{\mathbbm{1}_{\zeta >t}}]\le {{\mathbb{E}}_{\xi }}[\mathbf{X}_{1}^{{{\alpha }_{1}}}\ldots \mathbf{X}_{k}^{{{\alpha }_{k}}}{\mathbbm{1}_{\zeta >t}}]\le \prod_{j=1}^{k}{{{\mathbb{E}}_{\xi }}{{[{{\mathbf{X}}_{j}}{\mathbbm{1}_{\zeta >t}}]}^{{{\alpha }_{j}}}}}=\chi_{1}^{\alpha_1}\ldots\chi_{k}^{\alpha_k}.\]
The proof is then completed by \autoref{69211} (c).
\end{proof}
\section{Applications}\label{applications}
In this section, we apply the solution representation $\chi=\mathbb{E}\mathbf{X}$ of (FNS), where $\mathbf{X}$ is defined by \eqref{eq:822191}, and the majorizing principle to obtain some well-posedness results for the Navier-Stokes equations and ill-posedness results for the Montgomery-Smith equation.

\subsection{Initial data in majorization space}
For each standard majorizing kernel $h:\mathbb{R}^d\backslash\{0\}\to(0,\infty)$, we define a \emph{majorization space} $F_h$ as follows \begin{equation}\label{Fh}F_h=\left\{f\in L^1_{\text{loc}}(\mathbb{R}^d):\ \underset{\xi\in\mathbb{R}^d}{\text{esssup}\,}|f(\xi)|h(\xi)^{-1}<\infty\right\},\ \ \|f\|_{F_h}=\|{c_0f}/{h}\|_{L^\infty}.\end{equation}
We denote by $F_h(\mathbb{R}^d;\mathbb{C}^d)$, $F_h(\mathbb{R}^d;\mathbb{C})$,\ldots if reference to the codomain is needed. Denote $F_{h,\infty}=L^\infty((0,\infty);F_h)$.
\begin{prop}\label{decay}
Let $h:\mathbb{R}^d\backslash\{0\}\to(0,\infty)$ be a standard majorizing kernel and let $v_0\in F_h(\mathbb{R}^d;\mathbb{C}^d)$. We have the following conclusions.
\begin{enumerate}[(a)]
\item If $\|v_0\|_{F_h}\le 2$ then {\eqref{FNS}}$_{v_0}$ has a solution $v\in F_{h,\infty}$ with $\|v\|_{F_{h,\infty}}\le 2$.
\item If $\|v_0\|_{F_h}< 2$ then the global solution $v$ obtained in Part (a) has a decay 
\[v(\xi,t)\le C\frac{\mu\|v_0\|_{F_h}}{2-\mu\|v_0\|_{F_h}}e^{-\kappa|\xi|\sqrt{t}}h\ \ \ \forall\,(\xi,t),\] 
where $C>0$ is a universal constant, $\mu=(8e^{3/4})^\kappa$, and $\kappa>0$ is any number less than $\min\left\{1,\frac{4(\ln 2-\ln\|v_0\|_{F_h})	}{3(1+4\ln 2)}\right\}$.
\item Moreover, for any $T\in[0,\infty]$, this solution is unique in the class of solutions satisfying $\|v\|_{F_{h,\infty}(\mathbb{R}^d\times[0,T))}=\mathrm{esssup}_{t\in[0,T)}\|v(t)\|_{F_h}<2$.

\end{enumerate}
\end{prop}
\begin{proof}
To show (a), observe that $\|\chi_0\|_{L^\infty}=\|v_0\|_{F_h}\le 2$ where $\chi_0=c_0v_0/h$, and 
therefore, $\psi_0=|\chi_0|/2\in[0,1]$ a.e. Let $\lwrY=\mathbf{Y}\mathbbm{1}_{\zeta>t}$  be the minimal solution process for \eqref{nFMS}$_{\psi_0}$. On the event $[\zeta>t]$, $\lwrY$ is a product of product of values of $\psi_0$. Thus,
$\mathbb{E}_\xi|\lwrY|\le1$ for a.e.\ on $\mathbb{R}^d\times [0,\infty)$. Therefore, $\psi=\mathbb{E}_\xi \lwrY$ is the minimal cascade solution of 
\eqref{nFMS}$_{\psi_0}$ on $\mathbb{R}^d\times [0,\infty)$. Consequently, by the majorizing principle \autoref{69211}, 
the minimal cascade solution for \eqref{nFNS}$_{\chi_0}$, $\chi=\mathbb{E}_\xi \lwrX$, exists and satisfies
\[|\chi|\le2\psi\le 2\]
on a set $D\times[0,\infty)$, where $D$ is the set of full measure in $\mathbb{R}^d$. Therefore, $v=h\chi/c_0$ is a global solution to \eqref{FNS}$_{v_0}$ and $\|v\|_{F_{h,\infty}}\le 2$.

To show (b), let $\gamma=\|\chi_0/2\|_{L^\infty}<1$. By the majorizing principle, $|\chi(\xi,t)|\le2\psi(\xi,t)$ where  $\psi$ is the cascade solution to \eqref{nFMS}$_{\gamma}$. The function $\psi$ can be expressed as $\psi(\xi,t)=\sum_{n=1}^\infty\gamma^np_n(\xi,t)$ where
\begin{equation}\label{probability}
p_n(\xi,t)=\mathbb{P}_{\xi}(\zeta_\xi>t,\ \text{exactly~}n\text{~paths~cross~the~horizon~}t).\end{equation}
By conditioning on the first time of branching, one gets
\begin{equation}\label{probrecursion}
{{p}_{n}}(\xi ,t)=\int_{0}^{t}{|\xi {{|}^{2}}{{e}^{-s|\xi {{|}^{2}}}}\int_{{{\mathbb{R}}^{d}}}{\sum\limits_{k=1}^{n-1}{{{p}_{k}}(\eta ,t-s){{p}_{n-k}}(\xi -\eta ,t-s)H(\eta |\xi )d\eta ds}}}.\end{equation}
Note that $p_1(\xi,t)=e^{-|\xi|^2t}\le e^{1/4}e^{-|\xi|\sqrt{t}}$. By the elementary estimate ${{e}^{-|\xi {{|}^{2}}s}}{{e}^{-|\eta |\sqrt{t-s}}}{{e}^{-|\xi -\eta |\sqrt{t-s}}}\le {{e}^{1/2}}{{e}^{-|\xi {{|}^{2}}s/2}}{{e}^{-|\xi |\sqrt{t}}}$ and induction on $n\ge 1$, one can show that 
\[p_n(\xi,t)\le\theta\lambda^{n-1}C_ne^{-|\xi|\sqrt{t}}\ \ \ \forall\,n\in\mathbb{N},\]
where $\theta=e^{1/4}$, $\lambda=2e^{3/4}$, and $\{C_n\}$ is the Catalan sequence $C_1=1$, $C_n=\sum_{k=1}^{n-1}C_kC_{n-k}$. Note that $C_n$ grows asymptotically as $4^n$. Since $p_n\in[0,1]$, for any $\kappa\in(0,1)$,
\[{{p}_{n}}(\xi ,t)\le {{\left( \theta {{\lambda }^{n-1}}{{C}_{n}}{{e}^{-|\xi |\sqrt{t}}} \right)}^{\kappa }}\lesssim {{(4\lambda )}^{\kappa n}}{{e}^{-\kappa |\xi |\sqrt{t}}}.\]
Therefore, by choosing $\kappa$ small enough,
\[\psi (\xi ,t)\lesssim \sum\limits_{n=1}^{\infty }{{{({{4}^{\kappa }}{{\lambda }^{\kappa }}\gamma )}^{n}}{{e}^{-\kappa |\xi |\sqrt{t}}}}= \sum\limits_{n=1}^{\infty }{{{(\mu \gamma )}^{n}}{{e}^{-\kappa |\xi |\sqrt{t}}}}=\frac{\mu \gamma }{1-\mu \gamma }{{e}^{-\kappa |\xi |\sqrt{t}}}.\]

To show Part (c), we adapt the martingale-type argument used in \cite{lejan} to show uniqueness of the thinned cascade solutions and in \cite{alphariccati} to generate multiple solution in conditions of explosion.
Let $\tilde{v}(\xi,t)$ be a solution to \eqref{FNS}$_{v_0}$ satisfying 
\[\alpha=\|\tilde{v}\|_{F_{h,\infty}(\mathbb{R}^d\times[0,T))}<2,\] 
and denote $\tilde{\chi}=c_0\tilde{v}/h$. For $n\in\mathbb{N}$ and $t\in[0,T)$, consider a sequence of stochastic processes $\tilde{X}_n$ defined inductively by
\[
\tilde{X}_0(\xi,t)=\tilde{\chi}(\xi,t),\qquad \tilde{X}_{n+1}(\xi,t)=\chi_0(\xi)\mathbbm{1}_{T_0>t}+\tilde{X}_{n}^{(1)}(W_1,t-T_0)\spr{\xi}\tilde{X}_{n}^{(2)}(W_2,t-T_0)\mathbbm{1}_{T_0<t}.
\]
Note that for each $n$, $\tilde{X}_n$ is well-defined. Moreover, $\tilde{\chi}^{(n)}:=\mathbb{E}\tilde{X}_n$ satisfies the Picard's iteration
\begin{equation}\label{eq:71211-2}{\tilde{\chi}^{(0)}}=\tilde{\chi},\ \ \ {\tilde{\chi}^{(n+1)}}={U}(\chi_0)+\mathfrak{B}({\tilde{\chi}^{(n)}},{\tilde{\chi}^{(n)}})
\end{equation} 
where $U$ and $\mathfrak{B}$ as in \eqref{eq:71211}. (Recall, that by \autoref{531211}, $\chi=\mathbb{E}_\xi[{\mathbf{X}}]$ is the pointwise limit of the Picard iteration \eqref{eq:71211}.) 
It follows from an induction on $n$ that $\mathbb{E}\tilde{X}_n=\tilde{\chi}$ for all $n$. On the non-explosion event $[\zeta>t]$,  $\tilde{X}_n=\mathbf{X}=\mathbf{X}\mathbbm{1}_{\zeta>t}$ for sufficiently large $n$ (depending on $\omega$), while on the explosion event event $[\zeta<t]$, $\tilde{X}_n$ is a $\spr{}$-products of a combination of $\tilde{\chi}$ and $\chi_0$ taken in an appropriate order. Taking into account \eqref{sprod-bd}, we notice that $|\tilde{X_n}|$ is bounded from above by $\alpha^{N_n}/2^{N_{n}-1}$, where $N_n$ is the (random) number of paths in the cascade truncated to $n$ generations. On the event of explosion, $N_n\uparrow\infty$ a.s.\ as $n\to\infty$, we have $0=\mathbf{X}\mathbbm{1}_{\zeta>t}$. Thus, we have $\tilde{X_n}\to\mathbf{X}\mathbbm{1}_{\zeta>t}$ a.s.\ and therefore by Lebesgue's Dominated Convergence Theorem,
\[
\tilde{\chi}=\lim\limits_{n\to\infty}\mathbb{E}\tilde{X}_n=\mathbb{E}_\xi{\mathbf{X}}=\chi.
\]
Thus, the minimal cascade solution $\chi=\mathbb{E}_\xi{\mathbf{X}}$ is the unique solution to \eqref{nFNS}$_{\chi_0}$ satisfying $\mathrm{esssup}{|\chi(\xi,t)|}<2$. In conclusion, $v=\chi h$ is the unique solution to \eqref{FNS} satisfying 
$\|{v}\|_{F_{h,\infty}}<2$.
\end{proof}
\begin{rem}
The existence results in \cite{lejan}, and in \cite[Thm.\ 1.1]{rabi} requires a stronger smallness condition on the initial data, namely $\|v_0\|_{F_h}\le 1$, under which they also obtain the uniqueness of solutions in the ball $\{\|v\|_{F_{h,\infty}}\le 1\}$. Our improvement of the existence result by allowing $\|v_0\|_{F_h}\le 2$ is owing to the fact that we use a symmetrized product and do not use a thinning procedure. Note, however, that our uniqueness result does not include to the boundary of the open ball $\{\|v\|_{F_{h,\infty}}< 2\}$. It turns out that the uniqueness for \eqref{FMS} fails if one includes the ball's boundary (\autoref{nonuniqueness}).
\end{rem}
\subsection{Initial data in adapted spaces}\label{general}
The use of Banach fixed-point theorem to obtain a solution to the Navier-Stokes equations
has resulted in many fruitful theories concerning mild solutions as well as weak solutions. 
Typically, to obtain a solution to \eqref{FNS} on the time interval $(0,T)$, one chooses a space $E$ and a space $\mathcal{E}_T$ so that an initial data $v_0\in E$, with small norm if necessary, yields a convergent sequence $v^{(n)}$ in $\mathcal{E}_T$, where $v^{(n)}$ is defined by the Picard's iteration
\begin{equation}\label{eq:621211}{v^{(0)}}=0,\ {v^{(n+1)}}={U}(v_0)+{B}({v^{(n)}},{v^{(n)}}).\end{equation} 
To this end, it is sufficient to choose $E$ and $\mathcal{E}_T$ such that ${U}$ is a bounded linear operator from $E$ to $\mathcal{E}_T$, and ${B}$ is a bounded bilinear operator from $\mathcal{E}_T\times\mathcal{E}_T$ to $\mathcal{E}_T$ (see \cite[Ch.\ 15]{lemarie2002}). Apriori, the sequence $\{v^{(n)}(\xi,t)\}$ does not necessarily converge for almost every $(\xi,t)\in\mathbb{R}^d\times(0,T)$. For example, it is well-known that for small $v_0\in E=L^2(|\xi|^{d/2-1})$, or equivalently for small $u_0=\mathscr{F}^{-1}\{{v}_0\}$ in $\dot{H}^{d/2-1}$, the iteration \eqref{eq:621211} yields a convergent sequence $v^{(n)}$ in $\mathcal{E}_T=L^4_tL^2_\xi(|\xi|^{(d-1)/2})$, or equivalently a convergent sequence ${u}^{(n)}=\mathscr{F}^{-1}\{{v}^{(n)}\}$ in $L^4_t\dot{H}^{(d-1)/2}_x$ (\cite[Thm.\ 15.3]{lemarie2002}, \cite[Cor.\ 5.11]{bahouri}). 
Without a more detailed analysis of the equation, one can only conclude that $\{v^{(n)}\}$ has a subsequence that converges pointwise almost everywhere. Note that small initial data $v_0$ in a majorization space $F_h$ guarantees the pointwise convergence (\autoref{decay}).
The main purpose of this section is to give more general criteria on $v_0$, especially those that do not require pointwise smallness of $v_0$, such that the entire Picard's iteration \eqref{eq:621211} converges pointwise almost everywhere to a solution of \eqref{FNS} (\autoref{75211}). This solution is understood in the following sense.
\begin{defi}
Let $v_0:\mathbb{R}^d\to\mathbb{C}^d$ be a measurable function. A function $v=v(\xi,t)$ is an \emph{admissible solution} to \eqref{FNS}$_{v_0}$ on $\mathbb{R}^d\times(0,T)$ if there exists a subset $D\subset\mathbb{R}^d$ with full measure such that the following conditions are satisfied:
\begin{enumerate}[(i)]
\item $v$ is measurable on $D\times(0,T)$ and has the integrability property
\begin{equation}\label{intcond2}\int_{0}^{t}|\xi {{|}}{{e}^{|\xi {{|}^{2}}s}}{\int_{{{\mathbb{R}}^{d}}}{|{{v }}(\eta ,s){{\odot }_{\xi }}{{v }}(\xi-\eta ,s)|d\eta ds}<\infty }\ \ \ \forall\,(\xi,t)\in D\times(0,T).\end{equation}
\item $v$ satisfies \eqref{FNS}$_{v_0}$ on $D\times(0,T)$.
\end{enumerate}
\end{defi}
We outline the strategy as follows. First, one observes that the Picard's iteration \eqref{eq:621211} can be converted into the Picard's iteration \eqref{eq:71211} by the relation $v^{(n)}=h\chi^{(n)}/c_0$. Since $h$ is a.e.\ positive, the pointwise convergence of $\{v^{(n)}\}$ is equivalent to the pointwise convergence of $\{\chi^{(n)}\}$. 
Therefore, one only needs to focus on the Picard's iteration \eqref{eq:71211} of \eqref{nFNS}. Second, by the majorizing principle, the sequence $\{\chi^{(n)}\}$ converges pointwise almost everywhere on $\mathbb{R}^d\times(0,T)$ if \eqref{nFMS}$_{|\chi_0|}$ has a solution that is finite almost everywhere on $\mathbb{R}^d\times(0,T)$. The problem then turns into finding a suitable function setting for \eqref{nFMS}, or equivalently up to a normalization, a suitable function setting for \eqref{FMS}. This method allows the smallness of $|\chi_0|$ to be in some integral sense instead of pointwise (\autoref{75212}).
We proceed to define the function settings for \eqref{FMS} or \eqref{nFMS} that suit our purposes. In the definition below, for the sake of simplicity of language, the term ``measurable function'' refers to an equivalence class of measurable functions that are equal to each other almost everywhere.
\begin{defi}\label{adspace}
Let $X$ be a normed space of measurable functions from $\mathbb{R}^d$ to $\mathbb{C}$. For some $T\in(0,\infty]$, let $\mathcal{X}_T$ be a Banach space of measurable functions from $\mathbb{R}^d\times(0,T)$ to $\mathbb{C}$. We call $X$ an \emph{adapted space}, $\mathcal{X}_T$ a \emph{path space}, and the pair $(X,\mathcal{X}_T)$ an \emph{admissible setting} of the equation $v=F_1(v_0)+F_2(v,v)$ if $F_1$ is a bounded linear map from $X$ to ${\mathcal{X}}_T$ and $F_2$ is a bounded bilinear map from ${\mathcal{X}}_T\times {\mathcal{X}}_T$ to ${\mathcal{X}}_T$.
\end{defi}
It can be observed that an admissible setting to {\eqref{FMS}}, i.e.\ the equation $v=U(v_0)+\widetilde{B}(v,v)$ where $U$ and $\widetilde{B}$ are given by \eqref{eq:71213}, corresponds one-to-one to an admissible setting to {\eqref{nFMS}}, i.e.\ the equation $v=U(v_0)+\widetilde{\mathfrak{B}}(v,v)$ where $\widetilde{\mathfrak{B}}$ are given by \eqref{eq:71216}. Indeed, if $(X,\mathcal{X}_T)$ is an admissible setting to \eqref{FMS} then $(Y=\frac{c_0}{h}X,\mathcal{Y}_T=\frac{c_0}{h}\mathcal{X}_T)$ is an admissible setting to \eqref{nFMS} and vice versa. Here the norms on $Y$ and $\mathcal{Y}_T$ are given by 
\begin{equation}\label{newnorm}\|w\|_{Y}=\|hw/c_0\|_X\ \ \forall w\in Y,\ \ \ \|w\|_{\mathcal{Y}_T}=\|hw/c_0\|_{\mathcal{X}_T}\ \ \forall w\in \mathcal{Y}_T.\end{equation}
One has the following abstract lemma.
\begin{lem}\label{abslem}
Let $(X,\mathcal{X}_T)$ be an admissible setting of the equation $w=F(w):=F_1(w_0)+F_2(w,w)$. Suppose $\|F_1(w_0\|_{\mathcal{X}_T}<\frac{1}{4\|F_2\|}$. Then $F$ is a contraction mapping from $\bar{B}_R$ to $\bar{B}_R$ with \[R=\frac{1-\sqrt{1-4\|{{F}_{2}}\|\|{{F}_{1}}({{v}_{0}}){{\|}_{{{\mathcal{X}}_{T}}}}}}{2\left\| {{F}_{2}} \right\|}.\] Consequently, $F$ has a unique fixed point in $\bar{B}_R$ which is given by the limit in $\mathcal{X}_T$ of the sequence $w^{(0)}=0$, $w^{(n+1)}=F(w^{(n)})$.
\end{lem}
Here $\|F_2\|$ denotes the operator norm of $F_2$ and $\bar{B}_R$ denotes the closed ball of radius $R$ centered at 0 in $\mathcal{X}_T$. The proof of the lemma is an simple application of the Contraction Mapping Theorem and is skipped.
\begin{thm}\label{75211}
For $d\ge 3$, let $v_0:\mathbb{R}^d\to \mathbb{C}^d$ be a measurable function. We have the following statements.
\begin{enumerate}[(a)]
\item (Global solution) Let $(X,\mathcal{X}_\infty)$ be an admissible setting to {\eqref{FMS}}. Suppose that $|v_0|\in X$ and $\||{v}_0|\|_X<\frac{1}{4\|U\|\|\widetilde{B}\|}$. Then there exists a subset $D\subset \mathbb{R}^d$ with full measure such that the iteration \eqref{eq:621211} converges pointwise everywhere on $D\times(0,\infty)$ to an admissible solution of {\eqref{FNS}}$_{v_0}$.
\item (Local solution) Let $\{(X,\mathcal{X}_\tau)\}_{\tau>0}$ be a family of admissible settings to {\eqref{FMS}}. Suppose that $|v_0|\in X$ and $\lim_{\tau\to 0}\|U(|v_0|)\|_{\mathcal{X}_\tau}=0$. Then there exist a number $\delta>0$ and a subset $D\subset \mathbb{R}^d$ with full measure such that the iteration \eqref{eq:621211} converges pointwise everywhere on $D\times(0,\delta)$ to an admissible solution of \eqref{FNS}$_{v_0}$.
\end{enumerate}
\end{thm}
\begin{proof}
We first show part (a). By \autoref{abslem}, the iteration $w^{(0)}=0$, $w^{(n+1)}=U(|v_0|)+\widetilde{B}(w^{(n)},w^{(n)})$ has a limit $w\in\mathcal{X}_\infty$. Note that $w$ is finite almost everywhere since it is equal to a function from $\mathbb{R}^d\times(0,\infty)$ to $\mathbb{C}$ almost everywhere (\autoref{adspace}). Now introduce
\[\chi_0=\frac{c_0v_0}{h},\ \ \psi^{(n)}=\frac{c_0w^{{(n)}}}{h},\ \ \psi=\frac{c_0w}{h},\ \ \mathcal{Y}_\infty=\frac{c_0}{h}\mathcal{X}_\infty\]
where $h$ is the scale-invariant kernel. The sequence $\{\psi^{(n)}\}$ satisfies the iteration \eqref{eq:71217} and converges in $\mathcal{Y}_\infty$ to the solution $\psi$ of \eqref{nFMS}$_{|\chi_0|}$. By \autoref{69211} (a), $\psi$ is a cascade solution to \eqref{nFMS}$_{|\chi_0|}$. Since $h$ is positive a.e., $\psi$ is finite a.e.\ on $\mathbb{R}^d\times(0,\infty)$. By \autoref{69211} (d), \eqref{nFNS}$_{\chi_0}$ has a cascade solution $\chi$ on $D\times(0,\infty)$ where $D\subset\mathbb{R}^d$ is a subset with full measure.  By \autoref{819191}, $\chi$ satisfies the integrability condition \eqref{intcond}. Therefore, $v=h\chi/c_0$ satisfies \eqref{FNS}$_{v_0}$ on $D\times(0,\infty)$ and has the integrability \eqref{intcond2}. This completes the proof of part (a). To show part (b), we observe that there exists $\delta>0$ such that $\|U(|v_0|)\|_{\mathcal{X}_\delta}<\frac{1}{4\|\widetilde{B}\|}$. By \autoref{abslem}, the sequence $w^{(n)}$ introduced above has a limit $w\in\mathcal{X}_\delta$. From here, the proof is essentially the same as the proof of part (a).
\end{proof}
Combining \autoref{75211} with the generalized majorizing principle, we obtain the following corollary.
\begin{cor}\label{75212}
For $d\ge 3$, let $\chi_0:\mathbb{R}^d\to \mathbb{C}^d$ be a measurable function and $f\in S$, where $S$ is the set given by \eqref{eq:73211}. Put $\phi_0=f(|\chi_0|)$. We have the following statements.
\begin{enumerate}[(a)]
\item (Global solution) Let $(Y,\mathcal{Y}_\infty)$ be an admissible setting to \eqref{nFMS}. Suppose that $\phi_0\in Y$ and $\|\phi_0\|_Y<\frac{1}{4\|U\|\|\widetilde{B}\|}$. Then there exists a subset $D\subset \mathbb{R}^d$ with full measure such that the cascade solutions to {\eqref{nFNS}}$_{\chi_0}$ are well-defined on $D\times(0,\infty)$, coincide with each other, and are given as the pointwise limit of iteration \eqref{eq:71211}.
\item (Local solution) Let $\{(Y,\mathcal{Y}_\tau)\}_{\tau>0}$ be a family of admissible settings to {\eqref{nFMS}}. Suppose that $\phi_0\in Y$ and $\lim_{\tau\to 0}\|U(\phi_0)\|_{\mathcal{Y}_\tau}=0$. Then there exist a number $\delta>0$ and a subset $D\subset \mathbb{R}^d$ with full measure such that the cascade solutions to {\eqref{nFNS}}$_{\chi_0}$ are well-defined on $D\times(0,\delta)$, coincide with each other, and are given as the pointwise limit of iteration \eqref{eq:71211}.
\end{enumerate}
\end{cor}
\begin{proof}
By \autoref{abslem}, the iteration $\phi^{(0)}=0$, $\phi^{(n+1)}=U(\phi_0)+\widetilde{\mathfrak{B}}(\phi^{(n)},\phi^{(n)})$ has a limit $\phi\in\mathcal{Y}_\tau$, which is finite almost everywhere on $\mathbb{R}^d\times(0,\infty)$. By \autoref{69211} (a), $\phi$ is the cascade solution to \eqref{FMS}$_{\phi_0}$. By \autoref{614211},  both cascade solutions to \eqref{FNS}$_{\chi_0}$ are well-defined on $D\times(0,\infty)$, where $D\subset\mathbb{R}^d$ is a subset with full measure, and coincide with each other. The proof of part (b) is similar.
\end{proof}
We end this section by giving some examples of admissible settings to \eqref{FMS}. The homogeneous Herz spaces (introduced by Herz \cite{herz}) are a family of normed spaces $\dot{K}^\alpha_{p,q}(\mathbb{R}^d)$ defined by
\[\dot{K}_{p,q}^{\alpha }(\mathbb{R}^d)=\left\{ f\in L_{\text{loc}}^{p}({{\mathbb{R}}^{d}}\backslash \{0\}),{{\left\| f \right\|}_{\dot{K}_{p,q}^{\alpha }}}={{\left\| {{\left\{ {{\left\| |\xi {{|}^{\alpha }}f \right\|}_{{{L}^{p}}({{A}_{k}})}} \right\}}_{k\in \mathbb{Z}}} \right\|}_{{{l}^{q}}(\mathbb{Z})}}<\infty  \right\}\]
where $1\le p,q\le\infty$, $\alpha\in\mathbb{R}$ and $A_k=\{\xi\in\mathbb{R}^d:\,2^k\le |\xi|\le 2^{k+1}\}$. Under suitable ranges\footnote{Specifically, $1\le p\le\infty, 1<q\le\infty,\alpha<d-d/p$ or $1\le p\le\infty, q=1,\alpha\le d-d/p$ (\cite[Lemma 2.2]{tsutsui}).} of $p,q,\alpha$, functions in $\dot{K}_{p,q}^{\alpha }$ are tempered distributions,
and thus $\dot{K}_{p,q}^{\alpha }$ is the image under the Fourier transform (or inverse Fourier transform) of the well-known homogeneous Fourier-Besov space $F\dot{B}_{p,q}^{\alpha }$, whose definition will be given in \autoref{fourierbesov}. When there is a need to specify the codomain, we will write $\dot{K}_{p,q}^{\alpha }(\mathbb{R}^d;\mathbb{C})$ or $\dot{K}_{p,q}^{\alpha }(\mathbb{R}^d;\mathbb{C}^d)$.
\begin{prop}\label{list}\footnote{Strictly speaking, the references mentioned here show the boundedness of the bilinear map $B$ instead of $\widetilde{B}$. However, their proofs do not use any special structure of the circle-dot product other than the fact that $|a\odot_\xi b|\le |a||b|$. Therefore, the proof of the boundedness of $\widetilde{B}$ follows verbatim.}
For $d\ge 1$, $T\in(0,\infty]$, $1\le p,q\le\infty$, $\alpha=d-1-d/p$, let $X=\dot{K}_{p,q}^{\alpha }(\mathbb{R}^d;\mathbb{C})$. The pair $(X,\mathcal{X}_T)$ is an admissible setting to {\eqref{FMS}} in the following cases. 
\begin{enumerate}[(a)]
\item $\mathcal{X}_T=L^q((0,T),L^p(|\xi|^{\alpha+2/q}))$ where $1<p<\infty$, $2<q<\infty$. \textup{(\cite[Thm.\ 8.11]{lemarie2016})}
\item $\mathcal{X}_T=L^p(|\xi|^\alpha)L^\infty_t((0,T))$ and $\frac{d}{d-1}\le p=q<\infty$. \textup{(\cite[Thm.\ 8.12]{lemarie2016})}
\item $\mathcal{X}_T=L^2((0,T),L^1)$ and $p=1,q=2$. \textup{(\cite[Thm.\ 8.14]{lemarie2016})}
\item $\mathcal{X}_T=L^\infty((0,T),\dot{K}^{\alpha}_{p,q})\cap L^1((0,T),\dot{K}^{\alpha+2}_{p,q})$ where $1<p\le\infty$, $1\le q\le\infty$. \textup{(\cite{xiao2014,lizheng})}\\
If $1<p\le\infty$, $1\le q<\infty$ or $p=1$, $1\le q\le 2$ then $\lim_{\tau\to 0}\|U(v_0)\|_{\mathcal{X}_\tau}=0$ for all $v_0\in X$. \textup{(\cite[Thm.\ 1.1]{lizheng})}
\item $\mathcal{X}_T=L^\infty((0,T),\dot{K}^{\alpha}_{p,q})\cap L^1((0,T),\dot{K}^{\alpha+2}_{p,q})$ where $p=1$, $1\le q\le 2$. \textup{(\cite{cannonewu, xiao2014})}
\end{enumerate}
\end{prop}
Here $L^p(\phi)$ denotes the weighted Lebesgue space defined by
\[L^p(\phi)=\left\{f:\mathbb{R}^d\to\mathbb{C},\ \|f\|_{L^p(\phi)}=\left(\int_{\mathbb{R}^d}|f|^p\phi^pdx\right)^{1/p}<\infty\right\}.\]
\subsection{On the global well-posedness in Fourier-Besov spaces}\label{fourierbesov}
As mentioned in the previous section, $\dot{K}_{p,q}^{\alpha }$ is the image under the Fourier transform of the homogeneous Fourier-Besov space $F\dot{B}_{p,q}^{\alpha }$  defined by
\[F\dot{B}_{p,q}^{\alpha }(\mathbb{R}^d)=\left\{ f\in \frac{\mathcal{S}'}{\mathcal{P}}:\ \mathscr{F}\{f\}\in L_{\text{loc}}^{p}({{\mathbb{R}}^{d}}\backslash \{0\}),\ {{\left\| f \right\|}_{F\dot{B}_{p,q}^{\alpha }}}={{\left\| {{\left\{ {{\||\xi|^\alpha \mathscr{F}\{f\} \|}_{{{L}^{p}}(A_k)}} \right\}}_{k\in \mathbb{Z}}} \right\|}_{{{l}^{q}}(\mathbb{Z})}}<\infty  \right\}\]
where $\mathcal{S}'$ denotes the space of tempered distributions and $\mathcal{P}$ denotes the space of polynomials (i.e.\ the tempered distributions whose Fourier transforms are supported at the origin). The following inclusions are straightforward from the definition 
\begin{eqnarray*}&&F\dot{B}^\alpha_{p,q_1}\subset F\dot{B}^\alpha_{p,q_2} \ \ \forall\,1\le q_1\le q_2\le\infty,\\
&&F\dot{B}^{\alpha_1}_{p_1,q}\subset F\dot{B}^{\alpha_2}_{p_2,q}\ \ \forall\, 1\le p_2\le p_1\le\infty,\,\alpha_1+\frac{d}{p_1}=\alpha_2+\frac{d}{p_2}.
\end{eqnarray*}
The homogeneous Sobolev spaces $\dot{H}^s=F\dot{B}^s_{2,2}$, the homogeneous Fourier-Herz spaces $\dot{\mathcal{B}}^s_q=F\dot{B}^s_{1,q}$ introduced by Cannone and Wu \cite{cannonewu}, the space $\mathcal{X}^{-1}=F\dot{B}^{-1}_{1,1}$ introduced by Lei and Lin \cite{linlei2011}, and the spaces ${PM}^a=F\dot{B}^a_{\infty,\infty}$ introduced by Cannone and Karch \cite{cannonekarch} are special cases of the Fourier-Besov spaces. The global well-posedness of the \eqref{NS} with initial data in the scale-critical space $F\dot{B}_{p,q}^{d-1-d/p }$ has been studied by many authors, e.g.\ \cite{cannonewu, xiao2014, konieczny, lemarie2016}. 
The common strategy is to show the boundedness of the bilinear map $B$ on a suitably chosen path space. \autoref{list} contains a list of some known results. In this section, we combine the generalized majorizing principle (\autoref{614211}) with the global-wellposedness in the space $F\dot{B}^{-1}_{1,1}$, which was shown in \cite{linlei2011}, to give an alternative proof for the global-wellposedness of \eqref{NS} in the spaces $F\dot{B}^{d-1-d/p}_{p,q}$ with $1\le q\le p<\infty$. Our method exploits the  symmetry property
$|\chi_0|\to |\chi_0|^p$, $\mathbf{Y}\to \mathbf{Y}^p$ of \eqref{eq:819191} rather than direct estimates on the bilinear map. 

Because $\|u_0\|_{F\dot{B}^\alpha_{p,q}}=\|\mathscr{F}\{{u}_0\}\|_{\dot{K}^\alpha_{p,q}}$, the global well-posedness of \eqref{NS} in $F\dot{B}^\alpha_{p,q}$ is equivalent to the global well-posedness of \eqref{FNS} in $\dot{K}^\alpha_{p,q}$.
\begin{prop}\label{simpleproof}
For $d\ge 3$, $1\le q\le p<\infty$, $\alpha=d-1-d/p$, there exists $\varepsilon=\varepsilon_{p,d}>0$ such that if $v_0\in\dot{K}_{p,q}^\alpha(\mathbb{R}^d,\mathbb{C}^d)$ and $\|v_0\|_{\dot{K}_{p,q}^\alpha}<\varepsilon$ then for some subset $D\subset \mathbb{R}^d$ with full measure, the iteration \eqref{eq:621211} converges pointwise on $D\times(0,\infty)$ to an admissible solution of {\eqref{FNS}}$_{v_0}$. This solution belongs to the spaces $\mathcal{X}_T$ as listed in \autoref{list} \textup{(}depending on the range of $p$ and $q$\textup{)}.
\end{prop}
\begin{proof}
In \cite{linlei2011}, it was shown that $X=\dot{K}^{-1}_{1,1}(\mathbb{R}^d;\mathbb{C})$ and $\mathcal{X}_\infty=L^\infty((0,\infty),\dot{K}^{-1}_{1,1})\cap L^1_t((0,\infty),\dot{K}^1_{1,1})$ are a pair of adapted space and path space of \eqref{FMS}. Let $h$ be the scale-invariant kernel $h(\xi)=c_d|\xi|^{1-d}$ and let $Y=\frac{c_0}{h}X$ and $\mathcal{Y}_\infty=\frac{c_0}{h}\mathcal{X}_\infty$ be the normed spaces whose norms are given by \eqref{newnorm}. Then ($Y$,$\mathcal{Y}_\infty$) is an admissible setting of \eqref{nFMS}. Hence, there exists $\delta>0$ such that if $\psi_0\in Y$ and $\|\psi_0\|_Y<\delta$ then \eqref{nFMS}$_{\psi_0}$ has a solution $\psi\in \mathcal{Y}_\infty$. Now let $v_0\in\dot{K}_{p,q}^\alpha(\mathbb{R}^d,\mathbb{C}^d)$. Denote $\chi_0=c_0v_0/h$ and $\phi_0=|\chi_0|^p$. Then
\begin{eqnarray*}
{{\left\| {{\phi }_{0}} \right\|}_{Y}}={{\left\| \frac{h{{\phi }_{0}}}{{{c}_{0}}} \right\|}_{X}}=\frac{{{c}_{d}}}{{{c}_{0}}}\int_{{{\mathbb{R}}^{d}}}{|\xi {{|}^{-d}}{{\phi }_{0}}d\xi }&=&{{\left( \frac{{{c}_{0}}}{{{c}_{d}}} \right)}^{p-1}}\left\| {{v}_{0}} \right\|_{\dot{K}_{p,p}^{\alpha }}^{p}\le {{\left( \frac{{{c}_{0}}}{{{c}_{d}}} \right)}^{p-1}}\left\| {{v}_{0}} \right\|_{\dot{K}_{p,q}^{\alpha }}^{p}.
\end{eqnarray*}
Choose $\varepsilon=\delta^{1/p}(c_d/c_0)^{1/p'}$, where $1/p+1/p'=1$. Suppose $\|v_0\|_{\dot{K}_{p,q}^\alpha}<\varepsilon$. Then $\|\phi_0\|_Y<\delta$. This implies that \eqref{nFMS}$_{\phi_0}$ has a solution $\phi\in \mathcal{Y}_\infty$, which is finite almost everywhere on $\mathbb{R}^d\times(0,\infty)$. Applying \autoref{614211} for $f(x)=x^p$, we conclude that \eqref{nFNS}$_{\chi_0}$ has a cascade solution $\chi$ that is well-defined on $D\times(0,\infty)$ for some subset $D\subset\mathbb{R}^d$ with full measure, and is given by the pointwise limit of the iteration \eqref{eq:71211}. Moreover, $\chi$ has the integrability property \eqref{intcond}. Therefore, $v=h\chi/c_0$ is an admissible solution to \eqref{FNS}$_{v_0}$.
\end{proof}
\subsection{Nonuniqueness and blowup phenomena of the Montgomery-Smith equation}\label{nonuniqueness}
It is clear from \autoref{69211} (c) that the functions $\psi_1(\xi,t)=\mathbb{P}_\xi(\zeta>t)$ and $\psi_2(\xi,t)\equiv 1$ are solutions to \eqref{FMS}$_{\psi_0\equiv 1}$. Consequently, if for a given standard majorizing kernel $h$ the problem \eqref{FMS}$_{\psi_0\equiv 1}$ has a unique solution in the ball $\{\|\psi\|_{L^\infty}\le 1\}$, then the stochastic cascade is almost surely nonexplosive for every $\xi\in\mathbb{R}^d\backslash\{0\}$. The converse is also true (see \cite[Prop.\ 2.1]{chaos}). It was shown in \cite{part2} that the cascade corresponding to the scale-invariant kernel $h_{\text{in}}(\xi)=\pi^{-3}|\xi|^{-2}$ in $\mathbb{R}^3$ is almost surely explosive for every $\xi\neq 0$, i.e.\ $\mathbb{P}_\xi(\zeta<\infty)=1$. We summarize this observation as follows.
\begin{prop}
For $d=3$, the Cauchy problem {\eqref{MS}} with the scale-invariant initial data $u_0(x)=\frac{2}{\pi}\frac{1}{|x|}$ has at least two solutions: the time-decaying solution $u_1(x,t)=(\frac{2}{\pi})^{3/2}\mathscr{F}^{-1}\{|\xi|^{-2}\mathbb{P}_\xi(\zeta>t)\}$ and the time-independent solution $u_2(x,t)=u_0(x)$.
\end{prop} 
\begin{proof}
Observe that $\mathscr{F}\{u_1\}=\frac{\psi_1h}{c_0}$ and $\mathscr{F}\{u_2\}=\frac{\psi_2h}{c_0}$, where $h(\xi)=\pi^{-3}|\xi|^{-2}$.
\end{proof}
\begin{rem}It is not clear how this method can be adapted to the Navier-Stokes equations to show the nonuniqueness of solutions. The same arguments simply do not work if the product of scalars is replaced by the circle-dot of vectors. 
Interested readers may refer to \cite{jiasverak2014} for a discussion on the possible lack of uniqueness of scale-invariant solutions for large scale-invariant data. 
\end{rem}
For $d=3$, the stochastic cascade corresponding to the Bessel kernel $h_\text{b}(\xi)=\frac{1}{2\pi}e^{-|\xi|}|\xi|^{-1}$ was shown to be almost surely non-explosive in \cite{part1,part2} by probabilistic methods. Below we give an analytic proof of this fact which exploits the aforementioned connection between the uniqueness of solutions and the non-explosion of the associated stochastic cascade.
\begin{prop}\label{nonexplosion}
The stochastic cascade corresponding to the Bessel kernel is almost surely non-explosive for all $\xi\in\mathbb{R}^3\backslash\{0\}$.
\end{prop}
\begin{proof}
Let us consider the Cauchy problem \eqref{MS} with the initial data $u_0(x)=\frac{2}{1+|x|^2}$. It is clear that $u_0\in L^p(\mathbb{R}^3)$ for all $p>3/2$. One can also check that $\mathscr{F}\{u_0\}=h/c_0$ where $h(\xi)=\frac{1}{2\pi}e^{-|\xi|}|\xi|^{-1}$. Recall that a mild solution to \eqref{MS} is a solution obtained by applying Banach fixed-point theorem to the equation
\[u={{e}^{\Delta t}}{{u}_{0}}+\int_{0}^{t}{\sqrt{-\Delta }{{e}^{\Delta (t-s)}}{{u}^{2}}(s)ds}.\]
The integrand can be expressed as $G(t-s)*u^2(s)$ where $G$ is a kernel with a scaling property $G(\lambda x,\lambda^2 t)=\lambda^{-4}G(x,t)$ for all $\lambda>0$. By choosing $\lambda=t^{-1/2}$, one can write $G(x,t)={{t}^{-2}}\tilde{G}(x/\sqrt{t})$, where $\mathscr{F}\{\tilde{G}\}=|\xi|e^{-|\xi|^2}$, and obtain the estimate
\begin{equation}\label{gestimate}{{\left\| \nabla ^mG(t) \right\|}_{{{L}^{q}}({{\mathbb{R}}^{3}})}}\le {{C}_{q}}{{t}^{\frac{3}{2q}-\frac{m}{2}-2}}\ \ \ \forall\, 1\le q\le \infty,\ \forall\,m=0,1,2,\ldots\end{equation}
The regularity theory of mild solutions to \eqref{MS} is similar to that of (NSE). In particular, \eqref{MS} has a unique mild solution in the critical space $\cap_{0<T<T^*}L^5(\mathbb{R}^3\times (0,T))$ where $T^*\in(0,\infty]$ is the maximal time of existence (see e.g.\ \cite{kato1984, fabes72}, \cite[Prop.\ 4.2]{thesis}).
Because $\psi_1(\xi,t)=\mathbb{P}_\xi(\zeta>t)$ and $\psi_2(\xi,t)\equiv 1$ are solutions to \eqref{FMS}$_{\psi_0\equiv 1}$, $u_1=\mathscr{F}^{-1}\{\psi_1h/c_0\}$ and $u_2=\mathscr{F}^{-1}\{\psi_2h/c_0\}$ are solutions to \eqref{MS}. By Hausdorff-Young inequality,
\[{{\left\| {{u}_{k}}(t) \right\|}_{{{L}^{5}}}}\lesssim {{\left\| \mathscr{F}\{{{u}_{k}}(t)\} \right\|}_{{{L}^{5/4}}}}\lesssim {{\left\| h \right\|}_{{{L}^{5/4}}}}<\infty\ \ \ \forall k=1,2,\,\forall t>0.\]
By the uniqueness of solutions of \eqref{MS}, $u_1=u_2$ a.e. Thus, for every $\xi\in\mathbb{R}^3\backslash\{0\}$, $\psi_1(\xi,t)=\psi_2(\xi,t)=1$ for a.e.\ $t>0$. By the continuity in $t$ of $\psi_1$ (\autoref{69211} (c)), one obtains $\psi_1\equiv 1$.
\end{proof}
The case $u_0(x)=\frac{2}{1+|x|^2}$ corresponds to $\psi_0=c_0\hat{u}_0/h\equiv 1$, which is, at least at an intuitive level, a critical value of $\psi_0$ that guarantees the finiteness of the expectation  of
\begin{equation}\label{eq:920211}\mathbf{Y}(\xi ,t,\omega)=\prod\limits_{{v}\in V(\xi ,t)}{{{\psi }_{0}}({{W}_{{v}}})}.\end{equation}
Whether the solution issued from the initial data $u^{(a)}_{0}$ for $a>1$, where
\begin{equation}\label{u0a}u^{(a)}_{0}(x)=\frac{2a}{1+|x|^2}\end{equation}
exhibits finite-time blowup is an interesting question. Because $u_0^{(a)}$ belongs to the critical space $L^3$, the solution (denoted by $u^{(a)}$) exists and is unique in $L^5(\mathbb{R}^3\times(0,T))$ for some $T>0$. Hence, it is natural to study the blowup in the critical setting $L^5_{x,t}$. We say that a function $f(x,t)$ blows up at time $T\in(0,\infty]$ if $\|f\|_{L^5(\mathbb{R}^3\times(0,\tau))}\uparrow\infty$ as $\tau\uparrow T$.
\begin{prop}
For any $a>1$, the Cauchy problem {\eqref{MS}} on $\mathbb{R}^3$ with initial data $u_0^{(a)}$ given by \eqref{u0a} does not have a mild solution in $L^5(\mathbb{R}^3\times(0,T))$ for any $T\ge\max\left\{1,\frac{9(1+\ln 2)^2}{16(\ln a)^2}\right\}$.
\end{prop}
\begin{proof}
Let $h(\xi)=\frac{1}{2\pi}e^{-|\xi|}|\xi|^{-1}$. For each $a>0$, the initial data $u^{(a)}_0$ corresponds to $\psi_0^{(a)}=c_0\mathscr{F}\{u^{(a)}\}/h=a$. Let $\psi^{(a)}$ be the cascade solution to \eqref{nFMS}$_a$. Then the mild solution to \eqref{MS} with initial data $u_0^{(a)}$ is $u^{(a)}=\mathscr{F}^{-1}\{h\psi^{(a)}/c_0\}$. Due to the non-explosion of the Bessel cascade (\autoref{nonexplosion}), $\psi^{(a)}=\mathbb{E}\mathbf{Y}^{(a)}$ where $\mathbf{Y}^{(a)}$ is given by \eqref{eq:920211} with $\psi_0=\psi^{(a)}_0=a$. It is clear fom \eqref{eq:920211} that $\mathbf{Y}^{(a)}\mathbf{Y}^{(1/a)}=1$. Thus,
\[{{\psi }^{(a)}}{{\psi }^{(1/a)}}=\mathbb{E}{{\mathbf{Y}}^{(a)}}\mathbb{E}{{\mathbf{Y}}^{(1/a)}}\ge \left({\mathbb{E}\sqrt{{{\mathbf{Y}}^{(a)}}{{\mathbf{Y}}^{(1/a)}}}}\right)^{2}=1.\]
For $a>1$, we have $1/a<1$ and \autoref{decay} (b) implies $\psi^{(1/a)}\le C_ae^{-\kappa|\xi|\sqrt{t}}$ where $\kappa=\min\{1,\frac{4\ln a}{3(1+\ln 2)}\}$. Therefore,
\[\mathscr{F}\{{u}^{(a)}\}=\frac{h{{\psi }^{(a)}}}{{{c}_{0}}}\ge \frac{h}{{{c}_{0}}{{\psi }^{(1/a)}}}\ge C_a|\xi |^{-1}{{e}^{|\xi |(\kappa \sqrt{t}-1)}}.\]
In particular, $\mathscr{F}\{u^{(a)}\}\not\in L^1((0,\kappa^{-2}),L^2(\mathbb{R}^3))$, which implies $u^{(a)}\not\in L^1((0,\kappa^{-2}),L^2(\mathbb{R}^3))$ by Plancherel theorem. Now let $T^*\le\infty$ be the maximal time of existence of $u^{(a)}$. For any $0<t\le T<T^*$, we have
\begin{eqnarray*}
{{\left\| {{u}^{(a)}}(t) \right\|}_{{{L}^{5/2}}}}&\le& {{\left\| {{e}^{t\Delta }}u_{0}^{(a)} \right\|}_{{{L}^{5/2}}}}+\int_{0}^{t}{{{\left\| G(t-s)*{{u}^{(a)}}{{(s)}^{2}} \right\|}_{{{L}^{5/2}}}}ds}\\
&\lesssim& {{\left\| u_{0}^{(a)} \right\|}_{{{L}^{5/2}}}}+\int_{0}^{t}{\frac{1}{{{(t-s)}^{1/2}}}\left\| {{u}^{(a)}}(s) \right\|_{{{L}^{5}}}^{2}ds}\\
&\le& {{\left\| u_{0}^{(a)} \right\|}_{{{L}^{5/2}}}}+{{\left( \int_{0}^{t}{\frac{1}{{{(t-s)}^{5/6}}}ds} \right)}^{3/5}}\left\| {{u}^{(a)}} \right\|_{{{L}^{5}}({{\mathbb{R}}^{3}}\times (0,T))}^{2}.
\end{eqnarray*}
This implies $u^{(a)}\in L^\infty((0,T),L^{5/2}(\mathbb{R}^3))$ for any $T<T^*$. By \eqref{gestimate} and Young's inequality for convolution,
\begin{eqnarray*}
{{\left\| {{u}^{(a)}}(t) \right\|}_{{{L}^{2}}}}&\le& {{\left\| u_{0}^{(a)} \right\|}_{{{L}^{2}}}}+\int_{0}^{t}{{{\left\| G(t-s) \right\|}_{{{L}^{10/7}}}}\left\| {{u}^{(a)}}(s) \right\|_{{{L}^{5/2}}}^{2}ds}\\
&\lesssim& {{\left\| u_{0}^{(a)} \right\|}_{{{L}^{2}}}}+\left\| {{u}^{(a)}} \right\|_{{{L}^{\infty }}((0,T),{{L}^{5/2}})}^{2}\int_{0}^{t}{\frac{1}{{{(t-s)}^{19/20}}}ds}.
\end{eqnarray*}
As a consequence, $u^{(a)}\in L^\infty((0,T),L^{2}(\mathbb{R}^3))$ for any $T<T^*$. Combining this with the fact that $u^{(a)}\not\in L^1((0,\kappa^{-2}),L^2(\mathbb{R}^3))$, we conclude that $T^*\le \kappa^{-2}$.
\end{proof}
Note that a function on $\mathbb{R}^d$ is real-valued only if its Fourier transform is conjugate even. Thus, for an initial data $u_0$ to be real-valued, it is necessary that its Fourier transform is supported on a symmetric region about the origin. Montgomery-Smith constructed a real-valued function $u_0^*$ 
whose Fourier transform is nonnegative and compactly supported in a symmetric region about the origin such that the mild solution $u$ to \eqref{MS} with initial data $au_0^*$, for sufficiently large $a>0$, fails to be in any Triebel-Lizorkin or Besov spaces after some finite time $t_0$ \cite[Thm.\ 1]{smith}. His proof in fact shows a stronger result, namely $u(\cdot,t_0)\not\in \dot{B}^\alpha_{-\infty,\infty}$ for any $\alpha\in\mathbb{R}$. 
Below we give an alternate explanation of this result from a stochastic cascade perspective. Our method is to derive lower estimates on the probabilities of horizon crossing similar to \eqref{probability}. 

We first recall the definition of the homogeneous Besov space $\dot{B}^\alpha_{-\infty,\infty}$. Let $\phi:\mathbb{R}^d\to\mathbb{R}$ be a Schwartz function whose Fourier transform takes values in the interval $[0,1]$, is supported in the shell $A=\{1/4\le|\xi|\le 1\}$, and satisfies \[\sum_{j\in\mathbb{Z}}\mathscr{F}\{\phi\}(2^{j}\xi)=1\ \ \ \forall \xi\in\mathbb{R}^d\backslash\{0\}.\]
Put $\phi_j(x)=2^{dj}\phi(2^{j}x)$. Then the space $\dot{B}^\alpha_{-\infty,\infty}$ is defined as
\[\dot{B}_{-\infty ,\infty }^{\alpha }({{\mathbb{R}}^{d}})=\left\{ f\in L_{\text{loc}}^{1}({{\mathbb{R}}^{d}}):\ {{\left\| f \right\|}_{\dot{B}_{-\infty ,\infty }^{\alpha }}}=\underset{j\in \mathbb{Z}}{\mathop{\sup }}\,{{2}^{\alpha j}}{{\left\| {{\phi }_{j}}*f \right\|}_{{{L}^{\infty }}}}<\infty  \right\}.\]
It is worth noting that the choice $A=\{1/4\le|\xi|\le 1\}$ is not important in the definition of $\dot{B}^\alpha_{-\infty,\infty}$. One can replace $A$ by any shell centered at the origin and still obtains an equivalent definition.
\begin{prop}
Let $u_0:\mathbb{R}^d\to\mathbb{R}$, $d\ge 3$, be a function 
whose Fourier transform is real-valued, nonnegative on $\mathbb{R}^d\backslash\{0\}$, and bounded away from zero on some nonempty open subset of $\mathbb{R}^d$. Then there exists $a_0>0$ such that for any $a\ge a_0$, the Cauchy problem {\eqref{MS}} with initial data $au_0$ satisfies $u(\cdot,1)\not\in \dot{B}^\alpha_{-\infty,\infty}$ for any $\alpha\in\mathbb{R}$.
\end{prop}
\begin{proof}
Without loss of generality, one may suppose that $\mathscr{F}\{u_0\}(\xi)\ge 1$ for all $\xi\in D=D_0\cup(-D_0)$ where $D_0=2e_1+B_1$. Here $B_r$ denotes the ball of radius $r$ centered at the origin and $e_1=(1,0,...,0)\in\mathbb{R}^d$. We consider the stochastic cascade associate with the kernel $h(\xi)=c_d|\xi|^{1-d}$. 
Denote 
$q_n(\xi,t)=\mathbb{P}_{\xi}(\zeta_\xi>t,\ \text{exactly~}n\text{~paths~cross~the~horizon~}t,\ \text{all~cross~the~horizon~on~}D)$. To be precise, that all paths cross the horizon $t$ on $D$ means $W_v\in D$ for all $v\in V(\xi,t)$ where $V(\xi,t)$ is the set of vertices defined by \eqref{Vxit}. Let $\psi_0=c_0\mathscr{F}\{au_0\}/h$ and $\psi$ be the cascade solution to \eqref{nFMS}$_{\psi_0}$. Then
\begin{equation}\label{eq:928212}\psi (\xi ,t)\ge \sum\limits_{n=1}^{\infty }{{{a}^{n}}{{q}_{n}}(\xi ,t)}.\end{equation}
We have $q_1(\xi,t)=e^{-|\xi|^2t}\mathbbm{1}_D(\xi)$. By conditioning on the first time of branching, one gets
\begin{equation}\label{probrecursion2}
{{q}_{n}}(\xi ,t)=\int_{0}^{t}{|\xi {{|}^{2}}{{e}^{-s|\xi {{|}^{2}}}}\int_{{{\mathbb{R}}^{d}}}{\sum\limits_{k=1}^{n-1}{{{q}_{k}}(\eta ,t-s){{q}_{n-k}}(\xi -\eta ,t-s)H(\eta |\xi )d\eta ds}}}.\end{equation}
We first show that the number
\[{{\alpha }_{d}}:={{\left( \frac{1}{4} \right)}^{2}}\min \left\{ \underset{\xi \in A}{\mathop{\inf }}\,\int_{D\cap (\xi -D)}{H(\eta |\xi )d\eta },\ \ \underset{\xi \in A}{\mathop{\inf }}\,\int_{A\cap (\xi -A)}{H(\eta |\xi )d\eta } \right\}\]
is positive.
Note that $h(\xi)=\tilde{h}(|\xi|)$ where $\tilde{h}$ is a decreasing function. For $\xi\in A$ and $\eta,\xi-\eta\in A\cup D$, we have
\[h(\eta),h(\xi-\eta)\ge \tilde{h}(3),\ \ h(\xi)\le\tilde{h}(1/4),\ \ H(\eta |\xi )=\frac{h(\eta )h(\xi -\eta )}{|\xi |h(\xi )}\ge \frac{\tilde{h}(3)^2}{\tilde{h}(1/4)}.\] 
For $\xi\in A$, a lower bound on the Lebesgue measure of $A\cap (\xi -A)$ is obtained by the observation that (see \autoref{balls}):
\begin{figure}[h]
\centering
\includegraphics[scale=.65]{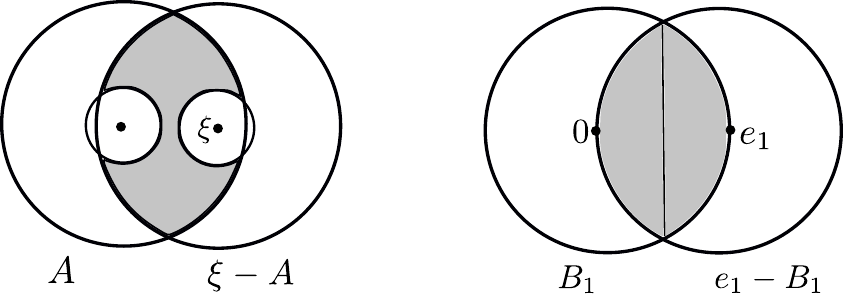}
\caption{Estimate of $m(A\cap (\xi -A))$.}
\label{balls}
\end{figure}
\begin{figure}[h]
\centering
\includegraphics[scale=.8]{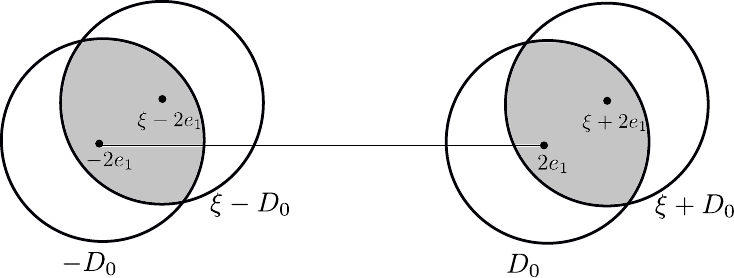}
\caption{Estimate of $m(D\cap (\xi -D))$.}
\label{balls2}
\end{figure}
\begin{eqnarray*}
m(A\cap (\xi -A))&\ge& m({{B}_{1}}\cap ({{e}_{1}}-{{B}_{1}}))-2m({{B}_{1/4}})\\
&=&2m(\{x=({{x}_{1}},...,{{x}_{d}})\in B_1:\ {{x}_{1}}\ge 1/2\})-2m({{B}_{1/4}})\\
&>&0.
\end{eqnarray*}
For $\xi\in A$, a lower bound on the Lebesgue measure of $D\cap (\xi -D)$ is obtained by the observation that (see \autoref{balls2}):
\[m(D\cap (\xi -D))=2m({{D}_{0}}\cap (\xi -{{D}_{0}}))\ge 4m(\{x=({{x}_{1}},...,{{x}_{d}})\in {{B}_{1}}:\ {{x}_{1}}\ge 1/2\})>0.\]
Therefore, 
\[{{\alpha }_{d}}\ge {{\left( \frac{1}{4} \right)}^{2}}\frac{\tilde{h}{{(1)}^{2}}}{\tilde{h}(1/4)}\min \left\{ \underset{\xi \in A}{\mathop{\inf }}\,m(D\cap (\xi -D)),\ \ \underset{\xi \in A}{\mathop{\inf }}\,m(A\cap (\xi -A)) \right\}>0.\]
Next, we show by induction that 
\begin{equation}\label{eq:928211}\tilde{q}_k(t):=\inf_{\xi\in A}q_{2^k}(\xi,t)\ge 2^{k+2}(\alpha_dt)^{2^k-1}(4e^t)^{-2^k} \ \ \forall\, k\in\mathbb{N},\,t>0.\end{equation}
For any $\xi\in A$,
\begin{eqnarray*}
{{q}_{2}}(\xi ,t)&=&\int_{0}^{t}{{{\left( \frac{1}{4} \right)}^{2}}{{e}^{-s}}\int_{{{\mathbb{R}}^{d}}}{{{q}_{1}}(\eta ,t-s){{q}_{1}}(\xi -\eta ,t-s)H(\eta |\xi )d\eta ds}}\\
&\ge& \int_{0}^{t}{{{\left( \frac{1}{4} \right)}^{2}}{{e}^{-s}}\int_{D\cap (\xi -D)}{{{e}^{-(t-s)}}{{e}^{-(t-s)}}H(\eta |\xi )d\eta ds}}\\
&\ge& {{\alpha }_{d}}\int_{0}^{t}{{{e}^{-s}}{{e}^{-2(t-s)}}ds}\ge {{\alpha }_{d}}t{{e}^{-2t}}.
\end{eqnarray*}
Hence, \eqref{eq:928211} is true for $k=1$. Suppose \eqref{eq:928211} is true for $k\ge 1$. We show that it is also true for $k+1$.
\begin{eqnarray*}{{q}_{{{2}^{k+1}}}}(\xi ,t)&\ge& \int_{0}^{t}{{{\left( \frac{1}{4} \right)}^{2}}{{e}^{-(t-s)}}\int_{{{\mathbb{R}}^{d}}}{{{q}_{{{2}^{k}}}}(\eta ,s){{q}_{{{2}^{k}}}}(\xi -\eta ,s)H(\eta |\xi )d\eta ds}}\\
&\ge& \int_{0}^{t}{{{\left( \frac{1}{4} \right)}^{2}}{{e}^{-(t-s)}}\int_{{{\mathbb{R}}^{d}}}{{{2}^{2(k+2)}}{{({{\alpha }_{d}}s)}^{{{2}^{k+1}}-2}}{{(4{{e}^{s}})}^{-{{2}^{k+1}}}}H(\eta |\xi )d\eta ds}}\\
&\ge& {{\alpha }_{d}}\alpha _{d}^{{{2}^{k+1}}-2}{{2}^{2k+4}}{{4}^{-{{2}^{k+1}}}}{{e}^{-t}}\int_{0}^{t}{{{s}^{{{2}^{k+1}}-2}}{{e}^{(1-{{2}^{k+1}})s}}ds}\\
&\ge& {{2}^{k+3}}{{({{\alpha }_{d}}t)}^{{{2}^{k+1}}-1}}{{(4e^t)}^{-{{2}^{k+1}}}}.
\end{eqnarray*}
Therefore, \eqref{eq:928211} is true for $k+1$. Applying this estimate to \eqref{eq:928212}, one gets

\begin{eqnarray*}
\psi (\xi ,t)\ge \sum\limits_{k=1}^{\infty }{{{a}^{{{2}^{k}}}}{{q}_{{{2}^{k}}}}(\xi ,t)}\ge \sum\limits_{k=1}^{\infty }{{{a}^{{{2}^{k}}}}{{2}^{k+2}}{{({{\alpha }_{d}}t)}^{{{2}^{k}}-1}}{{(4{{e}^{t}})}^{-{{2}^{k}}}}}\gtrsim \sum\limits_{k=1}^{\infty }{a{{2}^{k+2}}{{(4a{{\alpha }_{d}}t{{e}^{-t}})}^{{{2}^{k}}-1}}}.
\end{eqnarray*}
which is a convergent series when $a\ge a_0:=e/(4\alpha_d)$ and $t=1$. Suppose by contradiction that $u\in\dot{B}^\alpha_{-\infty,\infty}$ for some $\alpha\in\mathbb{R}$. Then
\[{{\left\| {{u}}(\cdot,1) \right\|}_{B_{-\infty ,\infty }^{\alpha }}}\ge {{\left\| {{\phi }_{0}}*{{u}}(\cdot,1) \right\|}_{{{L}^{\infty }}}}={{\left\| \mathscr{F}\{{{\phi }_{0}}\}\mathscr{F}\{{{u}}(\cdot,1)\} \right\|}_{{{L}^{1}}}}={{\left\| \mathscr{F}\{{{\phi }_{0}}\}\frac{h{{\psi }}(\cdot,1)}{{{c}_{0}}} \right\|}_{{{L}^{1}}}}=\infty. \]
This is a contradiction.
\end{proof}
\section*{Appendix}\label{appendix}
We now present an analytic proof for \autoref{614211}. For simplicity, we will only give the proof of a slightly weaker version of \autoref{614211}, which already contains the key technique. 
\begin{prop}\label{97191}
Let $\psi_0:\mathbb{R}^d\to[0,\infty]$ be a measurable function and $f\in S$, where $S$ is the set given by \eqref{eq:73211}. Let $\psi$ be the solution to {\eqref{nFMS}}$_{\psi_0}$ and $\phi$ be the solution to {\eqref{nFMS}}$_{\phi_0}$, where $\phi_0=f(\psi_0)$. Then $\psi\le f^{-1}(\phi)$.
\end{prop}
\begin{proof}
Recall that \eqref{nFMS}$_{\psi_0}$ has a solution $\psi:\mathbb{R}^d\times(0,\infty)\to[0,\infty]$ given by the pointwise limit of the nondecreasing sequence $\{\psi^{(n)}\}$:
\begin{equation}\label{eq:710211}
{{\psi}^{(0)}}=0,\ \ \ {{\psi}^{(n+1)}}=U(\psi_0)+\widetilde{\mathfrak{B}}({{\psi}^{(n)}},{{\psi}^{(n)}}),
\end{equation}
where $\widetilde{\mathfrak{B}}$ is given by \eqref{eq:71216}.
Let 
$\{\phi^{(n)}\}$ be a sequence given by ${{\phi}^{(0)}}=0$, ${{\phi}^{(n+1)}}=U(\phi_0)+\widetilde{\mathfrak{B}}({{\phi}^{(n)}},{{\psi}^{(n)}})$. 
Since $f$ is continuous and strictly increasing, it suffices to show that 
\begin{equation}\label{eq:710212}
\phi^{(n)}\ge f(\psi^{(n)})  \ \ \forall\, n\ge 0.
\end{equation}
We show by induction in $n$. The base case $n=0$ is obvious. Suppose \eqref{eq:710212} holds for some $n\ge 0$. Because $\int H(\eta|\xi)d\eta=1$,
\begin{eqnarray*}{{\phi }^{(n+1)}}(\xi ,t)={{e}^{-t|\xi {{|}^{2}}}}f(\psi_0(\xi))+\int_{0}^{t}{|\xi {{|}^{2}}{{e}^{-s|\xi {{|}^{2}}}}\int_{{{\mathbb{R}}^{d}}}{{{\phi }^{(n)}}(\eta ,t-s){{\phi }^{(n)}}(\xi -\eta ,t-s)H(\eta |\xi )d\eta }ds}\\
=\int_{0}^{\infty }{\int_{{{\mathbb{R}}^{d}}}{g(\eta,s)\left\{ f(\psi_0(\xi)){\mathbbm{1}_{s\ge t}}+{{\phi }^{(n)}}(\eta ,t-s){{\phi }^{(n)}}(\xi -\eta ,t-s){\mathbbm{1}_{s< t}} \right\}d\eta ds}}.
\end{eqnarray*}
where $g(\eta,s)=|\xi {{|}^{2}}{{e}^{-s|\xi {{|}^{2}}}}H(\eta |\xi )$. By the induction hypothesis and the submultiplicative property of $f$,
\begin{eqnarray*}
\phi^{(n+1)}(\xi,t)&\ge& \int_{0}^{\infty }{\int_{{{\mathbb{R}}^{d}}}{g(\eta,s)\left\{ f(\psi_0(\xi)){\mathbbm{1}_{s\ge t}}+f({{\psi }^{(n)}}(\eta ,t-s))f({{\psi }^{(n)}}(\xi-\eta ,t-s)){\mathbbm{1}_{s< t}} \right\}d\eta ds}}\\
&\ge&\int_{0}^{\infty }{\int_{{{\mathbb{R}}^{d}}}{g(\eta,s){{f\left( {{\psi }_{0}}(\xi ){\mathbbm{1}_{s\ge t}}+{{\psi }^{(n)}}(\eta ,t-s){{\psi }^{(n)}}(\xi -\eta ,t-s){\mathbbm{1}_{s< t}} \right)}}d\eta ds}}.
\end{eqnarray*}
Because $g(\eta,s)d\eta ds$ is a probability measure and $f$ is convex, one can apply Jensen's inequality:
\begin{eqnarray*}
{{\phi }^{(n+1)}}(\xi ,t)&\ge& f\left( \int_{0}^{\infty }{\int_{{{\mathbb{R}}^{d}}}{g(\eta ,s)({{\psi }_{0}}(\xi ){\mathbbm{1}_{s\ge t}}+{{\psi }^{(n)}}(\eta ,t-s){{\psi }^{(n)}}(\xi -\eta ,t-s){\mathbbm{1}_{s< t}})d\eta ds}} \right )\\
&=&f(\psi^{(n+1)}(\xi,t)).
\end{eqnarray*}
\end{proof}
\section{Acknowledgments}
A part of this work was conducted while TP was a postdoctoral fellow at Oregon State University.  The authors would like to thank our colleagues Christopher Orum and Edward Waymire for many stimulating conversations that motivated this research and clarified the exposition.  

\bibliography{References}
\end{document}